\documentclass[11pt, reqno]{amsart}

\usepackage{amsmath}
\usepackage{tabu}
\usepackage{tikz}
\usepackage[margin=1.2in]{geometry}
\usepackage[matrix,arrow,curve,frame]{xy}
\usepackage{hyperref}
\usepackage{amsthm}
\usepackage{amsfonts}
\usepackage{amssymb}
\usepackage{wasysym}
\usepackage{mathrsfs}
\usepackage{mathtools}

\definecolor{my-linkcolor}{rgb}{0.75,0,0}
\definecolor{my-citecolor}{rgb}{0.1,0.57,0}
\definecolor{my-urlcolor}{rgb}{0,0,0.75}
\hypersetup{
	colorlinks,
	linkcolor={my-linkcolor},
	citecolor={my-citecolor},
	urlcolor={my-urlcolor}
}

\title[Virasoro algebra at central charge $25$]{An $\mathfrak{sl}_2$-type tensor category for the Virasoro algebra at central charge $25$ and applications}
 \author{Robert McRae and Jinwei Yang}
\date{}

\address{(R. M.) Yau Mathematical Sciences Center, Tsinghua University, Beijing 100084, China}
  \email{rhmcrae@tsinghua.edu.cn}
  
  \address{(J. Y.) School of Mathematical Sciences, Shanghai Jiao Tong University, Shanghai 200240, China}
  \email{jinwei2@sjtu.edu.cn}
  

\newtheorem{thm}{Theorem}[section]
\newtheorem{cor}[thm]{Corollary}
\newtheorem{lem}[thm]{Lemma}
\newtheorem{prop}[thm]{Proposition}
\theoremstyle{definition}\newtheorem{defi}[thm]{Definition}
\theoremstyle{definition}\newtheorem{rem}[thm]{Remark}
\theoremstyle{definition}
\theoremstyle{definition}

\newcommand{\cE}{\mathcal{E}}
\newcommand{\cY}{\mathcal{Y}}

\newcommand{\cV}{\mathcal{V}}
\newcommand{\cA}{\mathcal{A}}
\newcommand{\cR}{\mathcal{R}}
\newcommand{\cM}{\mathcal{M}}

\newcommand{\cX}{\mathcal{X}}
\newcommand{\cF}{\mathcal{F}}

\newcommand{\cJ}{\mathcal{J}}
\newcommand{\cL}{\mathcal{L}}
\newcommand{\cO}{\mathcal{O}}

\newcommand{\cC}{\mathcal{C}}
\newcommand{\cD}{\mathcal{D}}

\newcommand{\cW}{\mathcal{W}}

\newcommand{\til}{\widetilde}

\newcommand{\CC}{\mathbb{C}}
\newcommand{\ZZ}{\mathbb{Z}}
\newcommand{\NN}{\mathbb{N}}

\newcommand{\QQ}{\mathbb{Q}}

\newcommand{\Id}{\mathrm{Id}}

\newcommand{\tens}{\boxtimes}
\newcommand{\vac}{\mathbf{1}}
\newcommand{\ind}{\mathrm{Ind}}
\DeclareMathOperator{\im}{Im}

\DeclareMathOperator{\rep}{Rep}
\let\ker\relax
\let\hom\relax
\DeclareMathOperator{\ker}{Ker}
\DeclareMathOperator{\hom}{Hom}

\newcommand{\repA}{\rep A}

\begin{document}
\bibliographystyle{alpha}

\numberwithin{equation}{section}

 \begin{abstract}
  Let $\mathcal{O}_{25}$ be the vertex algebraic braided tensor category of finite-length modules for the Virasoro Lie algebra at central charge $25$ whose composition factors are the irreducible quotients of reducible Verma modules. We show that $\cO_{25}$ is rigid and that its simple objects generate a semisimple tensor subcategory that is braided tensor equivalent to an abelian $3$-cocycle twist of the category of finite-dimensional $\mathfrak{sl}_2$-modules. We also show that this $\mathfrak{sl}_2$-type subcategory is braid-reversed tensor equivalent to a similar category for the Virasoro algebra at central charge $1$. As an application, we construct a simple conformal vertex algebra which contains the Virasoro vertex operator algebra of central charge $25$ as a $PSL_2(\mathbb{C})$-orbifold. We also use our results to study Arakawa's chiral universal centralizer algebra of $SL_2$ at level $-1$, showing that it has a symmetric tensor category of representations equivalent to $\mathrm{Rep}\,PSL_2(\mathbb{C})$.
  This algebra is an extension of the tensor product of Virasoro vertex operator algebras of central charges $1$ and $25$, analogous to the modified regular representations of the Virasoro algebra constructed earlier for generic central charges by I. Frenkel--Styrkas and I. Frenkel--M. Zhu.
 \end{abstract}

\maketitle

\tableofcontents

\section{Introduction}

The Virasoro Lie algebra is fundamental in the mathematics of two-dimensional conformal quantum field theory \cite{BPZ}. A major problem in the representation theory of the Virasoro algebra is the construction and study of conformal-field-theoretic braided tensor categories of modules at a fixed central charge (the scalar by which a canonical central element of the Virasoro algebra acts), or equivalently of modules for a Virasoro vertex operator algebra \cite{FZ1}. One especially wants to compute tensor product decompositions in such tensor categories and prove rigidity (existence of duals in a suitable sense). In this paper, we solve these problems for the Virasoro algebra at central charge $25$, building on our previous work \cite{MY} at certain negative rational central charges, and using the results of Orosz Hunziker \cite{OH} on vertex algebraic intertwining operators among Virasoro modules at central charge $25$.

Since the Virasoro algebra is the affine $W$-algebra associated to the affine Lie algebra $\widehat{\mathfrak{sl}}_2$ via quantum Drinfeld-Sokolov reduction \cite{FFr}, one expects close relations between the representation theories of the Virasoro algebra and of $\mathfrak{sl}_2$. Here, we confirm this expectation by showing that the Virasoro algebra at central charge $25$ admits a braided tensor category of modules that is equivalent to a twist of the category of finite-dimensional $\mathfrak{sl}_2$-modules by an abelian $3$-cocycle on $\ZZ/2\ZZ$. Moreover, we further confirm the duality between the Virasoro algebra at central charges $c$ and $26-c$ \cite{FF} in the case $c=25$  by showing that our $\mathfrak{sl}_2$-type tensor category for the Virasoro algebra at central charge $25$ is braid-reversed tensor equivalent to the one constructed at central charge $1$ in \cite{McR}.

In general, let $V_c$ be the simple Virasoro vertex operator algebra at central charge $c = 13-6t-6t^{-1}$ for $t \in \CC^\times$. It was shown in \cite{CJORY} that the category $\cO_c$ of $C_1$-cofinite grading-restricted generalized $V_c$-modules is the same as the category of finite-length $V_c$-modules whose composition factors are the irreducible quotients of reducible Verma modules, and that $\cO_c$ satisfies all conditions needed to apply Huang--Lepowsky--Zhang's (logarithmic) tensor category theory for vertex operator algebra module categories \cite{HLZ1}-\cite{HLZ8}. Thus $\cO_c$ has the braided tensor category structure described in \cite{HLZ8}. Moreover, rigidity and tensor products of simple objects in $\cO_c$ are known when $t\notin\QQ$ \cite{FZ2, CJORY} and when $t\in\QQ_{>0}$ \cite{Wa, Hu_Vir_tens,  Hu_rigid, Mi, McR, CMY2, MY}. Thus the remaining central charges to study are $c=13-6t-6t^{-1}$ with $t$ negative rational.

 In this paper, we focus on $t=-1$, that is, $c=25$. This case is different from other negative rational $t$ since $V_{25}$ is the quantum Drinfeld-Sokolov reduction of only one affine $\mathfrak{sl}_2$ vertex operator algebra (at level $-3$) rather than two. As a result, $\cO_{25}$ contains only one $\mathfrak{sl}_2$-related subcategory. Indeed, the simple objects of $\cO_{25}$ form a one-parameter family of modules $\cL_{r,1}$, $r\in\ZZ_+$, with lowest conformal weights $1-\frac{1}{4}(r+1)^2$. Our first main result is that the tensor category $\cO_{25}$ is rigid and that its simple modules have $\mathfrak{sl}_2$-type tensor products:
\begin{thm}\label{thm: main thm1}
Let $V_{25}$ be the simple Virasoro vertex algebra of central charge $25$, and let $\cO_{25}$ be the braided tensor category of $C_1$-cofinite grading-restricted generalized $V_{25}$-modules.
\begin{itemize}
\item[(1)] The category $\cO_{25}$ is rigid and ribbon, with duals given by the contragredient modules of \cite{FHL} and natural twist isomorphism $\theta = e^{2\pi i L_0}$.
\item[(2)] Tensor products of the irreducible modules in $\cO_{25}$ are as follows: for $r,r'\in\ZZ_+$,
\begin{equation*}
  \cL_{r,1}\tens\cL_{r',1}\cong\bigoplus_{\substack{k=\vert r-r'\vert+1\\ k+r+r'\equiv 1\,(\mathrm{mod}\,2)}}^{r+r'-1} \cL_{k,1}.
 \end{equation*}
\end{itemize}
\end{thm}

The proof of Theorem \ref{thm: main thm1} is based on a detailed analysis of the tensor generator $\cL_{2,1}$ using intertwining operators constructed in \cite{OH} and results on minimal conformal weight spaces of tensor product modules from \cite{MY}. The proof of rigidity (and self-duality) for $\cL_{2,1}$ requires somewhat involved computations with BPZ equations and intertwining operators since the vacuum space in $\cL_{2,1}\tens\cL_{2,1}\cong\cL_{1,1}\oplus\cL_{3,1}$ lies three levels above the tensor product module's lowest conformal weight space (which appears in the $\cL_{3,1}$ direct summand).

The category $\cO_{25}$ is not semisimple, since it contains arbitrary-length subquotients of Virasoro Verma modules and their contragredients, as well as non-split self-extensions of the modules $\cL_{r,1}$ for $r\geq 2$.
However, Theorem \ref{thm: main thm1}(2) shows that the semisimple subcategory $\cO_{25}^0$ consisting of all finite direct sums of the simple modules $\cL_{r,1}$, $r\in\ZZ_+$, is a tensor subcategory of $\cO_{25}$. By \cite[Theorem~$A_{\infty}$]{KW}, $\cO_{25}^0$ is tensor equivalent to the category $\cC(-1, \mathfrak{sl}_2)$ of finite-dimensional weight modules for Lusztig's modified form of the Drinfeld-Jimbo quantum group $U_q(\mathfrak{sl}_2)$ specialized at $q=-1$. Moreover, using the methods of \cite{GN}, we find that $\cC(-1, \mathfrak{sl}_2)$ has two braidings, and we determine which one makes $\cC(-1, \mathfrak{sl}_2)$ braided equivalent to $\cO_{25}^0$. There is a similar semisimple tensor subcategory $\cO_1^0\subseteq\cO_1$ of Virasoro modules at central charge $1$ \cite{McR}; we show that it is braided tensor equivalent to $\cC(-1, \mathfrak{sl}_2)$ equipped with the inverse braiding as for $\cO_{25}^0$. Thus $\cO_1^0$ and $\cO_{25}^0$ are braid-reversed equivalent:

\begin{thm}\label{thm: main thm2}
Let $X$ be the 2-dimensional standard representation in $\cC(-1, \mathfrak{sl}_2)$, and let $f_X = i_X \circ e_X$ be the composition of the coevaluation and evaluation maps for $X$. Then there is a braided tensor equivalence $\cO_{25}^0 \cong \cC(-1, \mathfrak{sl}_2)$, respectively $\cO_1^0\cong\cC(-1,\mathfrak{sl}_2)$, where $\cC(-1, \mathfrak{sl}_2)$ is equipped with the unique braiding $\cR$ such that $\cR_{X, X} = i(f_X - \Id_{X\otimes X})$, respectively $\cR_{X, X} = -i(f_X - \Id_{X\otimes X})$. In particular, $\cO_{1}^0$ is braid-reversed tensor equivalent to $\cO_{25}^0$.
\end{thm}

As an application of our results on $\cO_{25}^0$ (and its relation to $\cO_1^0$), we can construct interesting simple conformal vertex algebras containing $V_{25}$ as a subalgebra and describe their representation categories. These algebras have infinite-dimensional conformal weight spaces and no lower bound on their conformal weights, so we cannot study their representations by usual methods such as Zhu algebras. Instead, we use the vertex operator algebra extension theory of \cite{HKL, CKM1, CMY1} applied to the braided tensor category $\cO_{25}^0$.

The first conformal vertex algebra we construct has automorphism group $PSL_2(\CC)$ and contains $V_{25}$ as the $PSL_2(\CC)$-fixed-point vertex operator subalgebra. This is analogous to the realization of $V_{13-6p-6p^{-1}}$, $p\in\ZZ_+$, as the $PSL_2(\CC)$-fixed-point subalgebra of a vertex operator algebra $\cW(p)$, where $\cW(p)$ is the triplet $W$-algebra for $p\geq 2$ and the simple affine vertex operator algebra $L_1(\mathfrak{sl}_2)$ for $p=1$. Thus we call our new conformal vertex algebra $\cW(-1)$; we construct it by applying the braid-reversed tensor equivalence $\cO_1^0\rightarrow\cO_{25}^0$ to $L_1(\mathfrak{sl}_2)$, and the representation category of $\cW(-1)$ is braid-reversed tensor equivalent to the modular tensor category of $L_1(\mathfrak{sl}_2)$-modules:
\begin{thm}\label{thm:intro_W(-1)}
For $n\in\NN$, let $V(n)$ denote the $(n+1)$-dimensional irreducible $\mathfrak{sl}_2$-module.
\begin{enumerate}
\item There is a unique simple conformal vertex algebra $\cW(-1)$ such that
\begin{equation*}
\cW(-1)\cong\bigoplus_{n=0}^\infty V(2n)\otimes\cL_{2n+1,1}
\end{equation*}
as a $V_{25}$-module. Moreover, $PSL_2(\CC)$ acts faithfully on $\cW(-1)$ by automorphisms, and $\cW(-1)$ has the above decomposition as a $PSL_2(\CC)$-module.

\item The category of finite-length $\cW(-1)$-modules which as $V_{25}$-modules are direct sums of submodules $\cL_{r,1}$, $r\in\ZZ_+$, is a rigid semisimple braided tensor category which is braid-reversed equivalent to the modular tensor category of $L_1(\mathfrak{sl}_2)$-modules.
\end{enumerate}

\end{thm}

The second conformal vertex algebra we study contains $V_1\otimes V_{25}$ as a vertex operator subalgebra. It was constructed in \cite[Section 7]{Ar}, where it is called the chiral universal centralizer algebra of $SL_2$ at level $-1$ and denoted $\mathbf{I}_{SL_2}^{-1}$. For a general level $k$, the chiral universal centralizer $\mathbf{I}^k_{SL_2}$ can be obtained from the vertex algebra of chiral differential operators on $SL_2$ at level $k$ by a two-step process of quantum Drinfeld-Sokolov reduction; for $k\notin\QQ$ this was first done in \cite{FS}, where the algebras were called modified regular representations of the Virasoro algebra, and an alternate construction for $k\notin\QQ$ was given in \cite{FZ2}. For $k=-1$, the chiral universal centralizer $\mathbf{I}_{SL_2}^{-1}$ can also be constructed using the method of gluing vertex algebras from \cite{CKM2}; that is, $\mathbf{I}_{SL_2}^{-1}$ is the canonical algebra in the Deligne product of the braid-reversed tensor equivalent categories $\cO_1^0$ and $\cO_{25}^0$. Using the detailed tensor structure of $\cO_1^0$ and $\cO_{25}^0$ combined with vertex operator algebra extension theory, we can determine the representation theory of $\mathbf{I}_{SL_2}^{-1}$:
\begin{thm}\label{thm:intro_chiral_univ_center}
Let $\cL_{r,1}^{(i)}$, $r \in \ZZ_+$, denote the simple modules in $\cO_{i}^0$ for $i = 1, 25$ with lowest conformal weights $\frac{1}{4}(r-1)^2$ and $1-\frac{1}{4}(r+1)^2$, respectively.
\begin{enumerate}
\item There is a unique simple conformal vertex algebra $\mathbf{I}_{SL_2}^{-1}$ such that
\begin{equation*}
\mathbf{I}_{SL_2}^{-1}\cong\bigoplus_{r\in\ZZ_+} \cL_{r,1}^{(1)}\otimes\cL_{r,1}^{(25)}
\end{equation*}
as a $V_1\otimes V_{25}$-module.

\item The category of finite-length $\mathbf{I}_{SL_2}^{-1}$-modules which as $V_1\otimes V_{25}$-modules are direct sums of submodules $\cL_{r,1}^{(1)}\otimes\cL_{r',1}^{(25)}$, $r,r'\in\ZZ_+$, is a rigid semisimple symmetric tensor category equivalent to $\rep PSL_2(\CC)$.
\end{enumerate}
\end{thm}

We expect our results for $c=25$ to have analogues in the categories $\cO_{13-6t-6t^{-1}}$ for general negative rational $t$. We plan to study these categories in future papers. Especially, we expect that $V_{13+6p+6p^{-1}}$ for $p\in\ZZ_{\geq 2}$ can be realized as the $PSL_2(\CC)$-fixed point subalgebra of a simple conformal vertex algebra $\cW(-p)$ (since the first version of this paper was completed, this has now been proved in an updated version of \cite{McR-cosets}). Then a detailed description of the tensor category $\cO_{13+6p+6p^{-1}}$ should enable the determination of the category of $\cW(-p)$-modules, similar to Theorem \ref{thm:intro_W(-1)}(2). A detailed understanding of $\cO_{13+6p+6p^{-1}}$ should also allow determination of the representation theory of the chiral universal centralizer $\mathbf{I}_{SL_2}^{-2+1/p}$. In contrast to the $p=1$ case, we expect $\mathbf{I}_{SL_2}^{-2+1/p}$ for $p\geq 2$ to have interesting non-semisimple representation theory, since this occurs for $V_{13-6p-6p^{-1}}$.

 We remark that $\mathbf{I}_{SL_2}^k$ for $k\notin\QQ$ can already be studied using results from \cite{CJORY, FZ2}; we summarize the results on the representation categories of these algebras in Remark \ref{rem:generic_ch_univ_cent}. We also remark that similar vertex algebra extension methods can be used to study the representation theory of the vertex algebras of chiral differential operators for groups $G$ at levels $k$, at least in cases that the Kazhdan-Lusztig categories for the affine Lie algebra of $\mathfrak{g}=\mathrm{Lie}(G)$ at levels $k$ and $-k-2h^\vee$ (where $h^\vee$ is the dual Coxeter number of $\mathfrak{g}$) are known to be semisimple braided tensor categories.

The remaining contents of this paper are as follows. In Section \ref{sec:prelim}, we introduce the Virasoro category $\cO_{25}$ as well as some elements of the vertex algebraic braided tensor category structure on categories such as $\cO_{25}$. In Section \ref{sec:L21_tens_prods}, we derive some results on the structure of $\cL_{2,1}\tens\cL_{r,1}$, $r\in\ZZ_+$, in $\cO_{25}$, and we determine $\cL_{2,1}\tens\cL_{2,1}$ completely. In Section \ref{sec:L21_rigidity}, we prove that $\cL_{2,1}$ is rigid, and then in Section \ref{sec:tens_prods_and_rigidity}, we complete the proof of Theorem \ref{thm: main thm1}. In Section \ref{sec:O_25_0}, we introduce the tensor subcategory $\cO_{25}^0$ and prove Theorem \ref{thm: main thm2}. In Section \ref{sec:W(-1)} we construct the conformal vertex algebra $\cW(-1)$ and determine its representation category, and then we determine the representation theory of the chiral universal centralizer $\mathbf{I}_{SL_2}^{-1}$ in Section \ref{sec:chiral_univ_center}. We place results needed for some assertions in Theorems \ref{thm:intro_W(-1)} and \ref{thm:intro_chiral_univ_center} (especially the uniqueness assertions) in the appendices, since their proofs are somewhat long and technical and are not specifically related to the Virasoro algebra at $c=25$.

\medskip

\noindent\textbf{Acknowledgments.} We thank Tomoyuki Arakawa for pointing out to us his construction of the chiral universal centralizer algebra and for suggesting that we study its representations. We also thank Drazen Adamovi\'{c}, Naoki Genra, Andrew Linshaw, and the referee for comments. R. McRae is supported by a startup grant from Tsinghua University. J. Yang is supported by a startup grant from Shanghai Jiao Tong University.

\section{Preliminaries}\label{sec:prelim}

Let $\cV ir$ be the Virasoro Lie algebra with basis $\lbrace L_n\,\vert\,n\in\ZZ\rbrace\cup\lbrace\mathbf{c}\rbrace$, where $\mathbf{c}$ is central and
\begin{equation*}
 [L_m,L_n]=(m-n)L_{m+n}+\frac{m^3-m}{12}\delta_{m+n,0}\mathbf{c}
\end{equation*}
for $m,n\in\ZZ$. A $\cV ir$-module $W$ has \textit{central charge} $c\in\CC$ if $\mathbf{c}$ acts on $W$ by the scalar $c$. It is natural to parametrize central charges by $c=13-6t-6t^{-1}$ for $t\in\CC^\times$. From the Feigin--Fuchs criterion for singular vectors in Virasoro Verma modules \cite{FF}, a Verma module at such a central charge can be reducible only if its lowest conformal weight is
\begin{equation}\label{eqn:h_rs_def}
 h_{r,s}:=\frac{r^2-1}{4} t-\frac{rs-1}{2}+\frac{s^2-1}{4} t^{-1} = \frac{1}{4t}(tr-s)^2-\frac{(t-1)^2}{4t}
\end{equation}
for some $r,s\in\ZZ$. Let $\cV_{r,s}$ denote the Verma module at central charge $c$ with lowest conformal weight $h_{r,s}$. If $r,s\in\ZZ_+$, then $\cV_{r,s}$ is indeed reducible, with unique irreducible quotient $\cL_{r,s}$ (see \cite[Section 5.3]{IK} for more details). For $r,s\in\ZZ$, we will generally use $v_{r,s}$ to denote a generating vector in $\cV_{r,s}$ or $\cL_{r,s}$ of conformal weight $h_{r,s}$.

We will focus on the central charge $c=25$, corresponding to $t=-1$. In this case, the conformal weight symmetry $h_{r,s}=h_{r+1,s-1}$ implies that every conformal weight in \eqref{eqn:h_rs_def} equals $h_{r,1}$ for some $r\in\ZZ$. Then using also the conformal weight symmetry $h_{r,s}=h_{-r,-s}$, every conformal weight in \eqref{eqn:h_rs_def} equals a unique $h_{r,1}$ for $r\in\lbrace -1,-2\rbrace\cup\ZZ_+$. From \cite[Section 5.3]{IK}, the Verma modules $\cV_{-1,1}$, $\cV_{-2,1}$ are irreducible, and we have embedding diagrams
\begin{equation}\label{diag:Verma_embedding}
 \cV_{-i,1} \longrightarrow \cV_{i,1} \longrightarrow \cV_{i+2, 1} \longrightarrow \cV_{i+4,1} \longrightarrow \cV_{i+6,1} \longrightarrow \cdots 
\end{equation}
for $i=1,2$. In particular, the maximal proper submodule of $\cV_{r,1}$ for $r\in\ZZ_+$ is $\cV_{r-2,1}$, that is, $\cL_{r,1} =\cV_{r,1}/\cV_{r-2,1}$ for $r\in\ZZ_+$.

For any central charge $c\in\CC$, the universal Virasoro vertex operator algebra $V_c$ is the quotient of the Verma module $\cV_{1,1}$ (with lowest conformal weight space $\CC\vac$) by the submodule generated by the singular vector $L_{-1}\vac$. We are interested in the category $\cO_c$ of $C_1$-cofinite grading-restricted generalized $V_c$-modules. By \cite{CJORY}, $\cO_c$ equals the category of finite-length $\cV ir$-modules at central charge $c$ whose composition factors are irreducible quotients of reducible Verma modules, and $\cO_c$ admits the vertex algebraic braided tensor category structure of \cite{HLZ1}-\cite{HLZ8}. Thus at central charge $c=25$, the category $\cO_{25}$ consists of finite-length $\cV ir$-modules at central charge $25$ whose composition factors come from $\lbrace\cL_{r,1}\,\vert\,r\in\ZZ_+\rbrace$. 

The Verma modules $\cV_{r,1}$ are not objects of $\cO_{25}$, but all their proper quotients are, as are all proper submodules of the Verma module contragredients $\cV_{r,1}'$ for $r\in\ZZ_+$. Since all irreducible $\cV ir$-modules $\cL_{r,1}$ are self-contragredient, the surjections $\cV_{r,1}\rightarrow\cL_{r,1}$ dualize to injections $\cL_{r,1}\rightarrow\cV_{r,1}'$. In particular, $\cL_{r,1}$ is the $V_{25}$-submodule of $\cV_{r,1}'$ generated by the lowest conformal weight space. As a $\cV ir$-module, the vertex operator algebra $V_{25}$ is isomorphic to $\cL_{1,1}$ since the singular vector $L_{-1}\vac\in\cV_{1,1}$ generates the maximal proper submodule $\cV_{-1,1}$. In particular, $V_{25}$ is a simple self-contragredient vertex operator algebra.

We now recall some elements of the vertex algebraic braided tensor category structure on $\cO_{25}$, beginning with the definition of (logarithmic) intertwining operator among modules for any vertex operator algebra $V$ (see for example \cite{HLZ2}):
\begin{defi}
 Suppose $W_1$, $W_2$, and $W_3$ are grading-restricted generalized $V$-modules. An \textit{intertwining operator} of type $\binom{W_3}{W_1\,W_2}$ is a linear map
 \begin{align*}
  \cY: W_1\otimes W_2 & \rightarrow W_3[\log x]\lbrace x\rbrace\nonumber\\
   w_1\otimes w_2 & \mapsto \cY(w_1,x)w_2=\sum_{h\in\CC}\sum_{k\in\NN} (w_1)_{h,k} w_2\,x^{-h-1}(\log x)^k
 \end{align*}
which satisfies the following properties:
\begin{enumerate}
 \item \textit{Lower truncation}: For any $w_1\in W_1$, $w_2\in W_2$, and $h\in\CC$, $(w_1)_{h+n,k} w_2 =0$ for $n\in\ZZ$ sufficiently large, independently of $k$.
 
 \item The \textit{Jacobi identity}: For $v\in V$ and $w_1\in W_1$,
 \begin{align*}
  x_0^{-1}\delta\left(\frac{x_1-x_2}{x_0}\right) Y_{W_3}(v,x_1)\cY(w_1,x_2) & - x_0^{-1}\left(\frac{-x_2+x_1}{x_0}\right)\cY(w_1,x_2)Y_{W_2}(v,x_1)\nonumber\\
  & = x_1^{-1}\delta\left(\frac{x_2+x_0}{x_1}\right)\cY(Y_{W_1}(v,x_0)w_1,x_2).
 \end{align*}

 \item The \textit{$L_{-1}$-derivative property}: For $w_1\in W_1$,
 \begin{equation*}
  \cY(L_{-1} w_1,x)=\dfrac{d}{dx}\cY(w_1,x).
 \end{equation*}
\end{enumerate}
\end{defi}

We will use two consequences of the Jacobi identity in the case that $v$ is the conformal vector $\omega=L_{-2}\vac$: the \textit{commutator formula}
\begin{equation}\label{eqn:Vir_comm_form}
 L_n\cY(w_1,x) =\cY(w_1,x)L_n+\sum_{i\geq 0}\binom{n+1}{i} x^{n+1-i}\cY(L_{i-1} w_1,x)
\end{equation}
and the \textit{iterate formula}
\begin{align}\label{eqn:Vir_it_form}
\cY(L_n w_1,x) =\sum_{i\geq 0} (-1)^i\binom{n+1}{i}\left( L_{n+1-i}\,x^i\cY(w_1,x)+(-1)^n x^{n+1-i}\cY(w_1,x)L_{i-1}\right). 
\end{align}
Associated to any intertwining operator $\cY$ of type $\binom{W_3}{W_1\,W_2}$, we have an \textit{intertwining map}
\begin{equation*}
 I: W_1\otimes W_2\rightarrow\overline{W}_3 :=\prod_{h\in\CC} (W_3)_{[h]}
\end{equation*}
defined by 
\begin{equation*}
 I(w_1\otimes w_2) =\cY(w_1,1)w_2
\end{equation*}
for $w_1\in W_1$, $w_2\in W_2$, where we realize the substitution $x\mapsto 1$ using the real-valued branch of logarithm $\ln 1=0$. We then say that $\cY$ is \textit{surjective} if $W_3$ is spanned by the projections of $\cY(w_1,1)w_2$ to the conformal weight spaces of $W_3$, as $w_1$ ranges over $W_1$ and $w_2$ ranges over $W_2$. Note that the substitution $x\mapsto 1$ in \eqref{eqn:Vir_comm_form} and \eqref{eqn:Vir_it_form} immediately yields commutator and iterate formulas for intertwining maps.

Tensor products, also known as fusion products, of $V$-modules were defined in \cite{HLZ3} using intertwining maps; equivalently, we can use intertwining operators:
\begin{defi}
 Let $\cC$ be a category of grading-restricted generalized $V$-modules and let $W_1$, $W_2$ be modules in $\cC$. A \textit{tensor product} of $W_1$ and $W_2$ in $\cC$ is a pair $(W_1\tens W_2,\cY_\tens)$, with $W_1\tens W_2$ a module in $\cC$ and $\cY_\tens$ an intertwining operator of type $\binom{W_1\tens W_2}{W_1\,W_2}$, that satisfies the following universal property: For any module $W_3$ in $\cC$ and intertwining operator $\cY$ of type $\binom{W_3}{W_1\,W_2}$, there is a unique $V$-homomorphism $f: W_1\tens W_2\rightarrow W_3$ such that $\cY=f\circ\cY_\tens$.
\end{defi}

If the tensor product $(W_1\tens W_2, \cY_\tens)$ exists, then the tensor product intertwining operator $\cY_\tens$ is surjective by \cite[Proposition 4.23]{HLZ3}. Also following the notation of \cite{HLZ3}, we will use the notation $w_1\tens w_2:=\cY_\tens(w_1,1)w_2\in\overline{W_1\tens W_2}$. Tensor product intertwining operators $\cY_\tens$ can be used to characterize the full braided tensor category structure on categories such as $\cO_{25}$, including the left and right unit isomorphisms $l$ and $r$, the associativity isomorphisms $\cA$, and the braiding isomorphisms $\cR$. For a detailed description of these isomorphisms, see \cite{HLZ8} or the exposition in \cite[Section 3.3]{CKM1}.

For the category $\cO_{25}$ of $C_1$-cofinite $V_{25}$-modules, all intertwining operators among the modules $\cL_{r,1}$ were constructed in \cite{OH}. In particular, there is a non-zero (and one-dimensional) space of intertwining operators of type $\binom{\cL_{r'',1}}{\cL_{r,1}\,\cL_{r',1}}$ if and only if 
\begin{equation*}
 r''\in\lbrace r+r'-1, r+r'-3,\ldots, \vert r-r'\vert+1\rbrace.
\end{equation*}
This strongly suggests but does not quite prove the $\mathfrak{sl}_2$-type fusion product formula
\begin{equation*}
 \cL_{r,1}\tens\cL_{r',1}\cong\bigoplus_{\substack{k=\vert r-r'\vert+1\\ k+r+r'\equiv 1\,(\mathrm{mod}\,2)}}^{r+r'-1} \cL_{k,1}
\end{equation*}
for $r,r'\in\ZZ_+$. We will prove this formula in the following sections.

\section{Tensor products involving \texorpdfstring{$\cL_{2,1}$}{L{2,1}}}\label{sec:L21_tens_prods}

In this section, we study how $\cL_{2,1}$ tensors with the irreducible modules in $\cO_{25}$ using results from \cite{OH, MY}. To begin, we fix an $\NN$-grading convention for modules in $\cO_{25}$. Let $W$ be a grading-restricted generalized $V_{25}$-module, and let $I$ denote the set of cosets $i\in\CC/\ZZ$ such that for some $h\in i$, the conformal weight space $W_{[h]}$ is non-zero. For $i\in I$, let $h_i$ be the minimal conformal weight of $W$ in $i$. Then the $\NN$-grading $W=\bigoplus_{n\in\NN} W(n)$, where $W(n)=\bigoplus_{i\in I}W_{[h_i+n]}$, satisfies
\begin{equation}\label{eqn:N-grade_reln}
 v_m\cdot W(n)\subseteq W(\mathrm{wt}\,v+n-m-1)
\end{equation}
for homogeneous $v\in V_{25}$, $m\in\ZZ$, and $n\in\NN$. There may be multiple $\NN$-gradings of $W$ that satisfy \eqref{eqn:N-grade_reln}, but this is the $\NN$-grading we shall use.

For $n\in\NN$, we use $\pi_n: W\rightarrow W(n)$ to denote the projection with respect to the $\NN$-grading on $W$. We then obtain the following from Propositions 3.1 and 3.4 of \cite{MY}:
\begin{prop}\label{prop:basic_L21_fusion}
 Let $W$ be a $V_{25}$-module in $\cO_{25}$ such that there is a surjective intertwining operator $\cY$ of type $\binom{W}{\cL_{2,1}\,\cL_{r,1}}$ for some $r\in\ZZ_+$. Then
 \begin{equation*}
  W(0)=\mathrm{span}\lbrace \pi_0(\cY(v_{2,1},1)v_{r,1}), \pi_0(\cY(L_{-1}v_{2,1},1)v_{r,1})\rbrace,
 \end{equation*}
where $v_{2,1}\in\cL_{2,1}$ and $v_{r,1}\in\cL_{r,1}$ are generating vectors of minimal conformal weight, and the vectors
\begin{equation*}
 (1\pm r+(1\pm 1))\pi_0(\cY(v_{2,1},1)v_{r,1})-2\,\pi_0(\cY(L_{-1}v_{2,1},1)v_{r,1})
\end{equation*}
are, if non-zero, $L_0$-eigenvectors with $L_0$-eigenvalues $h_{r\pm1, 1}$.
\end{prop}

The preceding proposition especially applies to $W=\cL_{2,1}\tens\cL_{r,1}$ and $\cY=\cY_\tens$. We can now determine $(\cL_{2,1}\tens\cL_{r,1})(0)$ precisely using the main theorem of \cite{OH}, which asserts that there is a non-zero intertwining operator of type $\binom{\cL_{r+1,1}}{\cL_{2,1}\,\cL_{r,1}}$, and therefore there is a surjective homomorphism
\begin{equation*}
 \cL_{2,1}\tens\cL_{r,1}\longrightarrow\cL_{r+1,1}.
\end{equation*}
This means $(\cL_{2,1}\tens\cL_{r,1})_{[h_{r+1,1}]}\neq 0$. Moreover, since $h_{r-1,1}-h_{r+1,1}=r+1\in\ZZ_+$, Proposition \ref{prop:basic_L21_fusion} and our $\NN$-grading convention then imply that $(\cL_{2,1}\tens\cL_{r,1})(0)=(\cL_{2,1}\tens\cL_{r,1})_{[h_{r+1,1}]}$ and this space is one-dimensional. The universal property of induced $\cV ir$-modules then implies we have a homomorphism
\begin{equation*}
 \Pi_{r}^+: \cV_{r+1,1}\longrightarrow\cL_{2,1}\tens\cL_{r,1}
\end{equation*}
which is an isomorphism on degree-$0$ spaces.

Now if $r>1$, the main theorem of \cite{OH} also implies we have a surjective homomorphism
\begin{equation*}
 \cL_{2,1}\tens\cL_{r,1}\longrightarrow\cL_{r-1,1}.
\end{equation*}
When $r>1$, then, $\Pi_{r}^+$ is not surjective onto $\cL_{2,1}\tens\cL_{r,1}$ because $\cL_{r-1,1}$ is not a quotient of $\cV_{r+1,1}$ (although it is a subquotient). Thus, applying Proposition \ref{prop:basic_L21_fusion} to the natural surjective intertwining operator of type $\binom{\cL_{2,1}\tens\cL_{r,1}/\im\,\Pi_{r}^+}{\cL_{2,1}\,\,\,\cL_{r,1}}$, we see that the degree-$0$ space of $\cL_{2,1}\tens\cL_{r,1}/\im\,\Pi_{r}^+$ is one-dimensional and spanned by an $L_0$-eigenvector with eigenvalue $h_{r-1,1}$. Thus we get a homomorphism
\begin{equation*}
 \Pi_{r}^-: \cV_{r-1,1}\longrightarrow\cL_{2,1}\tens\cL_{r,1}/\im\,\Pi_{r}^+
\end{equation*}
which this time is surjective (because from Proposition \ref{prop:basic_L21_fusion}, the degree-$0$ subspace of $(\cL_{2,1}\tens\cL_{r,1}/\im\Pi_{r}^+)/\im\Pi_r^-$ must vanish).

We have now found that  for $r>1$, there is an exact sequence
\begin{equation}\label{eqn:c25_tens_prod_exact}
 0\longrightarrow \cV_{r+1,1}/\cJ_+ \longrightarrow\cL_{2,1}\tens\cL_{r,1}\longrightarrow\cV_{r-1,1}/\cJ_-\longrightarrow 0,
\end{equation}
where $\cJ_\pm=\ker\Pi_{r}^\pm$ and both Verma module quotients must be non-zero and $C_1$-cofinite. For the rest of this section, we focus on the case $r=2$:
\begin{thm}\label{thm:L21_L21_tens}
 In the tensor category $\cO_{25}$, $\cL_{2,1}\tens\cL_{2,1}\cong\cL_{1,1}\oplus\cL_{3,1}$.
\end{thm}
\begin{proof}
From the Verma module embedding diagram \eqref{diag:Verma_embedding}, the only non-zero $C_1$-cofinite quotient of $\cV_{1,1}$ is $\cV_{1,1}/\cV_{-1,1}\cong\cL_{1,1}$. So in the $r=2$ case of the exact sequence \eqref{eqn:c25_tens_prod_exact}, we have $\cJ_-=\cV_{-1,1}$. Similarly, we have either $\cJ_+=\cV_{1,1}$ or $\cJ_+=\cV_{-1,1}$. Then the existence of a surjection $\cL_{2,1}\tens\cL_{2,1}\rightarrow\cL_{3,1}$ means that we have another exact sequence
\begin{equation*}
 0\longrightarrow K\longrightarrow\cL_{2,1}\tens\cL_{2,1}\longrightarrow\cL_{3,1}\longrightarrow 0
\end{equation*}
where the submodule $K$ is either $\cL_{1,1}$ or a length-$2$ self-extension of $\cL_{1,1}$. Either way, $K$ is semisimple because $\cL_{1,1}$ does not admit non-split self-extensions (see \cite[Section 5.4]{GK}). This means any vector in the non-empty set $K\setminus(\im\Pi_{2,1}^+\cap K)$ generates a submodule isomorphic to $\cL_{1,1}$. Since $\cL_{1,1}$ is simple, such a submodule is complementary to $\im\Pi_{2,1}^+$, showing that
\begin{equation*}
 \cL_{2,1}\tens\cL_{2,1}\cong\cL_{1,1}\oplus(\cV_{3,1}/\cJ_+).
\end{equation*}
We have not yet shown that $\cV_{3,1}/\cJ_+=\cL_{3,1}$, but note that we have now shown that $\cL_{2,1}\tens\cL_{2,1}$ is not a logarithmic module, since $L_0$ acts semisimply on both $\cL_{1,1}$ and $\cV_{3,1}/\cJ_+$.

For showing that $\cJ_+=\cV_{1,1}$, we observe that the Verma submodule $\cV_{1,1}\subseteq\cV_{3,1}$ is generated by the degree-$3$ singular vector
\begin{equation*}
\widetilde{v}_{1,1} = (L_{-1}^3+4L_{-1}L_{-2}+2L_{-3})v_{3,1},
\end{equation*}
where $v_{3,1}$ is a singular vector generating $\cV_{3,1}$. Consequently, to complete the proof of the theorem, it is enough to show that $\Pi_2^+(\widetilde{v}_{1,1})=0$. To compute this vector in $(\cL_{2,1}\tens\cL_{2,1})(3)$, we begin with a lemma:
\begin{lem}\label{lem:nonzero}
 The vector $\pi_0(v_{2,1}\tens v_{2,1})\in(\cL_{2,1}\tens\cL_{2,1})(0)$ is non-zero.
\end{lem}
\begin{proof}
 From Proposition \ref{prop:basic_L21_fusion}, $(\cL_{2,1}\tens\cL_{2,1})(0)=(\cL_{2,1}\tens\cL_{2,1})_{[-3]}$ is spanned by
 \begin{equation*}
  5\,\pi_0(v_{2,1}\tens v_{2,1})-2\,\pi_0(L_{-1}v_{2,1}\tens v_{2,1}),
 \end{equation*}
while 
\begin{equation*}
 -\pi_0(v_{2,1}\tens v_{2,1})-2\,\pi_0(L_{-1}v_{2,1}\tens v_{2,1})=0.
\end{equation*}
The second expression shows that $\pi_0(v_{2,1}\tens v_{2,1})=0$ if and only if $\pi_0(L_{-1}v_{2,1}\tens v_{2,1})=0$. But the first shows that $\pi_0(v_{2,1}\tens v_{2,1})$ and  $\pi_0(L_{-1}v_{2,1}\tens v_{2,1})$ cannot both vanish, forcing $\pi_0(v_{2,1}\tens v_{2,1})\neq 0$.
\end{proof}

We may now assume $\Pi_2^+(v_{3,1})=\pi_0(v_{2,1}\tens v_{2,1})$, so that we need to show
\begin{equation}\label{eqn:sing_vect_vanish}
 (L_{-1}^3+4L_{-1}L_{-2}+2L_{-3})\pi_0(v_{2,1}\tens v_{2,1})=0.
\end{equation}
For this, we use another lemma:
\begin{lem}\label{lem:diff_eq}
 The series $\cY_\boxtimes(v_{2,1},x)v_{2,1}$ satisfies the formal differential equation
 \begin{equation*}
  \left( x^2\frac{d^2}{dx^2}-x\frac{d}{dx}-\frac{5}{4}\right)\cY_\boxtimes(v_{2,1},x)v_{2,1}=-\sum_{i\geq 1} L_{-i} x^i\cY_\boxtimes(v_{2,1},x)v_{2,1}.
 \end{equation*}
\end{lem}
\begin{proof}
 We use the $L_{-1}$-derivative property, the singular vector $(L_{-1}^2+L_{-2})v_{2,1}\in\cV_{2,1}$, the iterate formula for intertwining operators, the $L_{-1}$-commutator formula, and $h_{2,1}=-\frac{5}{4}$:
 \begin{align*}
  \frac{d^2}{dx^2} & \cY_\boxtimes(v_{2,1}, x)  v_{2,1}\nonumber\\  &=\cY_\tens(L_{-1}^2 v_{2,1},x)v_{2,1} =-\cY_\tens(L_{-2} v_{2,1},x)v_{2,1}\nonumber\\
  & =-\sum_{i\geq 0} (-1)^i\binom{-1}{i}\left(L_{-2-i} x^i\cY_\boxtimes(v_{2,1},x)v_{2,1}+x^{-1-i}\cY_\tens(v_{2,1},x)L_{i-1}v_{2,1}\right)\nonumber\\
& =  -\sum_{i\geq 2} \left(L_{-i} x^{i-2}\cY_\tens(v_{2,1},x)v_{2,1}\right)-\cY_\tens(v_{2,1},x)\left(x^{-1}L_{-1}+x^{-2}L_0\right)v_{2,1}\nonumber\\
& =-\sum_{i\geq 1} L_{-i} x^{i-2}\cY_\tens(v_{2,1},x)v_{2,1}+x^{-1}\frac{d}{dx}\cY_\tens(v_{2,1},x)v_{2,1}+\frac{5}{4} x^{-2}\cY_\tens(v_{2,1},x)v_{2,1}.
 \end{align*}
The lemma then follows from multiplying both sides by $x^2$ and rearranging terms.
\end{proof}

Now since $\cL_{2,1}\tens\cL_{2,1}$ is not logarithmic, the $L_0$-conjugation formula implies we can write
\begin{align*}
 \cY_\tens(v_{1,2},x)v_{1,2}& = x^{L_0-2h_{2,1}}(v_{2,1}\tens v_{2,1})\nonumber\\ 
 &=\sum_{n\geq 0} \pi_n(v_{2,1}\tens v_{2,1})\, x^{n+h_{3,1}-2h_{2,1}} =\sum_{ n\geq 0}\pi_n(v_{2,1}\tens v_{2,1})\,x^{n-1/2},
\end{align*}
with $\pi_0(v_{2,1}\tens v_{2,1})\neq 0$ by Lemma \ref{lem:nonzero}. If we insert this series into the differential equation of Lemma \ref{lem:diff_eq} and extract coefficients of powers of $x$, we get
\begin{equation*}
 n(n-3)\pi_n(v_{2,1}\tens v_{2,1}) =-\sum_{i=1}^n L_{-i} \pi_{n-i}(v_{2,1}\tens v_{2,1})
\end{equation*}
for all $n\in\NN$. Taking $n=1,2$ first, we get
\begin{align*}
 \pi_1(v_{2,1}\tens v_{2,1}) =\frac{1}{2}L_{-1}\pi_0(v_{2,1}\tens v_{2,1}),\quad \pi_2(v_{2,1}\tens v_{2,1}) = \left(\frac{1}{4}L_{-1}^2 +\frac{1}{2}L_{-2} \right)\pi_0(v_{2,1}\tens v_{2,1}).
\end{align*}
Then taking $n=3$, we get
\begin{align*}
 0 & = -L_{-1}\pi_2(v_{2,1}\tens v_{2,1})-L_{-2}\pi_1(v_{2,1}\tens v_{2,1})-L_{-3}\pi_0(v_{2,1}\tens v_{2,1})\nonumber\\
 & = \left(-\frac{1}{4}L_{-1}^3-\frac{1}{2}(L_{-1}L_{-2}+L_{-2}L_{-1})-L_{-3}\right)\pi_0(v_{2,1}\tens v_{2,1})\nonumber\\
 & =\left(-\frac{1}{4}L_{-1}^3-L_{-1}L_{-2}-\frac{1}{2}L_{-3}\right)\pi_0(v_{2,1}\tens v_{2,1}).
\end{align*}
This proves \eqref{eqn:sing_vect_vanish} and thus also the theorem.
\end{proof}

\section{Rigidity for \texorpdfstring{$\cL_{2,1}$}{L{2,1}}}\label{sec:L21_rigidity}

In this section, we show that the self-contragredient $V_{25}$-module $\cL_{2,1}$ is rigid and self-dual. The proof is similar to those in \cite[Section 4.2]{CMY2} and \cite[Section 4.1]{MY}, but it is more computationally intensive because $0$ is not the lowest conformal weight of $\cL_{2,1}\tens\cL_{2,1}$.
\begin{thm}\label{thm:L21_rigid}
 The irreducible module $\cL_{2,1}$ in $\cO_{25}$ is rigid and self-dual.
\end{thm}
\begin{proof}
The direct sum decomposition of $\cL_{2,1}\tens\cL_{2,1}$ from Theorem \ref{thm:L21_L21_tens} shows that for $s=1,3$, we have maps $i_s:\cL_{s,1}\rightarrow\cL_{2,1}\tens\cL_{2,1}$ and $p_s:\cL_{2,1}\tens\cL_{2,1}\rightarrow\cL_{s,1}$ such that
\begin{equation}\label{eqn:c25_coev_ev}
 p_s\circ i_{s'}=\delta_{s,s'}\Id_{\cL_{s,1}}\quad\text{and}\quad i_1\circ p_1+i_3\circ p_3=\Id_{\cL_{2,1}\tens\cL_{2,1}}.
\end{equation}
We take $i_1$ and $p_1$ as preliminary candidates for coevaluation and evaluation, respectively. Then it suffices to show that the rigidity composition
\begin{equation*}
 \cL_{2,1}\xrightarrow{r^{-1}}\cL_{2,1}\tens\cL_{1,1}\xrightarrow{\Id\tens i_1}\cL_{2,1}\tens(\cL_{2,1}\tens\cL_{2,1})\xrightarrow{\cA}(\cL_{2,1}\tens\cL_{2,1})\tens\cL_{2,1}\xrightarrow{p_1\tens\Id}\cL_{1,1}\tens\cL_{2,1}\xrightarrow{l}\cL_{2,1}
\end{equation*}
is non-zero (see for example Lemma 4.2.1 and Corollary 4.2.2 of \cite{CMY3}). Since $\cL_{2,1}$ is simple, the rigidity composition is a scalar multiple $\mathfrak{R}\cdot\Id_{\cL_{2,1}}$, so we need to show that $\mathfrak{R}\neq 0$.

For $s=1,3$, we define intertwining operators $\cY^s_{22}:=p_s\circ\cY_\boxtimes$ of type $\binom{\cL_{s,1}}{\cL_{2,1}\,\cL_{2,1}}$ and 
\begin{equation*}
 \cY^2_{2s}:= l\circ(p_1\tens\Id)\circ\cA\circ(\Id\tens i_s)\circ\cY_\tens=[l\circ(p_1\tens\Id)\circ\cA\circ\cY_\tens]\circ(\Id\otimes i_s)
\end{equation*}
of type $\binom{\cL_{2,1}}{\cL_{2,1}\,\cL_{s,1}}$. In particular,
$$\cY^2_{21}=\mathfrak{R}\cdot(r\circ\cY_\boxtimes)=\mathfrak{R}\cdot\Omega(Y_{\cL_{2,1}}),$$
where $\Omega(Y_{\cL_{2,1}})$ is the intertwining operator of type $\binom{\cL_{2,1}}{\cL_{2,1}\,\cL_{1,1}}$ related to the vertex operator $Y_{\cL_{2,1}}$ by skew-symmetry. We now use \eqref{eqn:c25_coev_ev} and the definitions of the unit and associativity isomorphisms in $\cO_{25}$ to calculate
\begin{align*}
 \mathfrak{R} \cdot\langle v_{2,1}, &\,\Omega(Y_{\cL_{2,1}})(v_{2,1},1)\cY^1_{22}(v_{2,1},x)v_{2,1}\rangle +\langle v_{2,1},\cY^2_{23}(v_{2,1},1)\cY^3_{22}(v_{2,1},x)v_{2,1}\rangle\nonumber\\
 & =\langle v_{2,1},\cY^2_{21}(v_{2,1},1)\cY^1_{22}(v_{2,1},x)v_{2,1}\rangle+\langle v_{2,1},\cY^2_{23}(v_{2,1},1)\cY^3_{22}(v_{2,1},x)v_{2,1}\rangle\nonumber\\
 & =\langle v_{2,1},[l\circ(p_1\tens\Id)\circ\cA\circ\cY_\boxtimes](v_{2,1},1)\cY_\tens(v_{2,1},x)v_{2,1}\rangle\nonumber\\
 & =\langle v_{2,1},[l\circ(p_1\tens\Id)\circ\cY_\tens](\cY_\tens(v_{2,1},1-x)v_{2,1},x)v_{2,1}\rangle\nonumber\\
 & =\langle v_{2,1}, Y_{\cL_{1,2}}(\cY_{22}^1(v_{2,1},1-x)v_{2,1},x)v_{2,1}\rangle,
\end{align*}
where we may take $x$ to be, for example, any real number in the interval $(\frac{1}{2},1)$. We can rescale $p_1$ (and correspondingly $i_1$) if necessary to ensure the following:
\begin{align}
\langle v_{2,1}, \Omega(Y_{\cL_{2,1}})(v_{2,1},1)\cY^1_{22}(v_{2,1},x)v_{2,1}\rangle &\in x^{-2h_{2,1}}\big( 1+ x\,\CC[[x]]\big),\label{phi2}\\
\langle v_{2,1},\cY^2_{23}(v_{2,1},1)\cY^3_{22}(v_{2,1},x)v_{2,1}\rangle & \in x^{h_{3,1}-2h_{2,1}}\CC[[x]],\label{phi1}\\
 \langle v_{2,1}, Y_{\cL_{1,2}}(\cY_{22}^1(v_{2,1},1-x)v_{2,1},x)v_{2,1}\rangle &\in (1-x)^{-2h_{2,1}}\left(1+\frac{1-x}{x}\CC\left[\left[\frac{1-x}{x}\right]\right]\right),\label{psi}
\end{align}
where $h_{2,1}=-\frac{5}{4}$ and $h_{3,1}=-3$.

Now, as first shown in \cite{BPZ} (see also \cite{Hu_Vir_tens, TW, CMY2}), products and iterates of intertwining operators such as \eqref{phi2}, \eqref{phi1}, and \eqref{psi} are solutions to a differential equation.  For central charge $25=13-6(-1)-6(-1)^{-1}$, with $h_{2,1}=-\frac{5}{4}$, the specific equation satisfied by \eqref{phi2}, \eqref{phi1}, and \eqref{psi} is 
\begin{equation*}
 x(1-x)\phi''(x)-(1-2x)\phi'(x)-\frac{5}{4} x^{-1}(1-x)^{-1}\phi(x)=0;
\end{equation*}
for a detailed derivation, see for example \cite[Proposition 4.1.2]{CMY2}, where one should set $p$ and $h_{1,2}$ (in the notation of that proposition) equal to $-1$ and $-\frac{5}{4}$, respectively. If we write
\begin{equation*}
 \phi(x)=x^{-2h_{2,1}}(1-x)^{-2h_{2,1}}f(x),
\end{equation*}
then $f(x)$ satisfies the hypergeometric differential equation
\begin{equation*}
 x(1-x)f''(x)+4(1-2x)f'(x)-10f(x)=0.
\end{equation*}
One solution to the hypergeometric equation near the singularity $x=0$ is
\begin{equation*}
 f_1(x)=x^{-3}+x^{-2},
\end{equation*}
and then because the differential equation is symmetric under the interchange $x\leftrightarrow 1-x$, the second fundamental solution is
\begin{equation*}
f_2(x) = \frac{1}{2}(1-x)^{-3}+\frac{1}{2}(1-x)^{-2} =\frac{1-\frac{1}{2}x}{(1-x)^3}.
\end{equation*}
Thus the original differential equation has the following basis of solutions near $x=0$:
\begin{align*}
\phi_1(x)& =x^{5/2}(1-x)^{5/2}(x^{-3}+x^{-2}) =x^{-1/2}(1-x)^{5/2}(1+x)\\
\phi_2(x)& =x^{5/2}(1-x)^{5/2}\cdot\frac{1-\frac{1}{2}x}{(1-x)^3} = x^{5/2}(1-x)^{-1/2}\left(1-\frac{1}{2}x\right).
\end{align*}
Then it is clear that \eqref{phi2} is $\phi_2(x)$ and \eqref{phi1} is some linear combination $a\,\phi_1(x)+b\,\phi_2(x)$, so
\begin{equation*}
 \langle v_{2,1}, Y_{\cL_{1,2}}(\cY_{22}^1(v_{2,1},1-x)v_{2,1},x)v_{2,1}\rangle =a\,\phi_1(x)+(\mathfrak{R}+b)\phi_2(x)
\end{equation*}
for $x$ such that both sides converge. Then from \eqref{psi}, the iterate on the left side here is a solution $\psi(x)$ to the differential equation such that $(1-x)^{-5/2}\psi(x)$ is analytic near $x=1$, with constant term $1$ when expanded as a series in $\frac{1-x}{x}$. The only solution with this property is $\frac{1}{2}\phi_1(x)$, so we may conclude that $a=\frac{1}{2}$ and $\mathfrak{R}+b=0$.

Now to show $\mathfrak{R}\neq 0$ and thus prove that $\cL_{2,1}$ is rigid, we need to compute $b$ and show that it is non-zero. We first note that from the $L_0$-conjugation formula,
\begin{align*}
 \frac{1}{2}\phi_1(x)+b\,\phi_2(x) & =\left\langle v_{2,1},\cY_{23}^2(v_{2,1},1)\cY_{22}^3(v_{2,1},x)v_{2,1}\right\rangle\nonumber\\
 & =\left\langle v_{2,1},\cY^2_{23}(v_{2,1},1)x^{L_0-2h_{2,1}}\cY^3_{22}(v_{2,1},1)v_{2,1}\right\rangle\nonumber\\
 & =\sum_{n\geq 0}\left\langle v_{2,1},\cY^2_{23}(v_{2,1},1)\pi_n(\cY^3_{22}(v_{2,1},1)v_{2,1})\right\rangle x^{n-1/2}.
\end{align*}
Denote $c_n=\left\langle v_{2,1},\cY^2_{23}(v_{2,1},1)\pi_n(\cY^3_{22}(v_{2,1},1)v_{2,1})\right\rangle$; from the series expansions of $\phi_1(x)$ and $\phi_2(x)$, we have $c_0=\frac{1}{2}$ and $c_3=\frac{25}{32}+b$. To determine $b$, we will compute $c_3$ from $c_0$ using properties of intertwining operators.

We first compute $\pi_3(\cY_{22}^3(v_{2,1},1)v_{2,1})$ in terms of $v_{3,1}=\pi_0(\cY_{22}^3(v_{2,1},1)v_{2,1})$. Let $\langle\cdot,\cdot\rangle$ now denote the unique invariant bilinear form on $\cL_{3,1}$ such that $\langle v_{3,1},v_{3,1}\rangle =1$; it is nondegenerate on the subspace $(\cL_{3,1})(3)=(\cL_{3,1})_{[0]}$ which has basis 
$$\lbrace L_{-3}v_{3,1},L_{-1}L_{-2}v_{3,1}\rbrace$$ 
since $(L_{-1}^3+4L_{-1}L_{-2}+2L_{-3})v_{3,1}=0$ in $\cL_{3,1}$. Now using the commutator formula \eqref{eqn:Vir_comm_form},
\begin{align*}
 \left\langle L_{-3} v_{3,1}, \cY^3_{22}(v_{2,1},1)v_{2,1}\right\rangle & =\left\langle v_{3,1}, L_3\cY^3_{22}(v_{2,1},1)v_{2,1}\right\rangle\nonumber\\
 & =\left\langle v_{3,1},\cY^3_{22}((L_{-1}+4L_0)v_{2,1},1)v_{2,1}\right\rangle\nonumber\\
 & =\left\langle v_{3,1},\cY^3_{22}((L_{-1}+L_0)v_{2,1},1)v_{2,1}\right\rangle +3h_{2,1}\left\langle v_{3,1},\cY^3_{22}(v_{2,1},1)v_{2,1}\right\rangle\nonumber\\
 & =\left\langle v_{3,1},[L_0,\cY^3_{22}(v_{2,1},1)]v_{2,1}\right\rangle+3h_{2,1}\langle v_{3,1},v_{3,1}\rangle\nonumber\\
 & =h_{3,1}+2h_{2,1} = -\frac{11}{2}
\end{align*}
and similarly,
\begin{align*}
 \left\langle L_{-1}L_{-2}v_{3,1},\cY^3_{22}(v_{2,1},1)v_{2,1}\right\rangle & =\left\langle L_{-2}v_{3,1},\cY^3_{22}((L_{-1}+2L_0)v_{2,1},1)v_{2,1}\right\rangle\nonumber\\
 & =\left\langle L_{-2}v_{3,1},\left[(\mathrm{ad}\, L_0+h_{2,1})\cY^3_{22}(v_{2,1},1)\right]v_{2,1}\right\rangle\nonumber\\
 &=(h_{3,1}+2)\left\langle L_{-2} v_{3,1},\cY^3_{22}(v_{2,1},1)v_{2,1}\right\rangle\nonumber\\
 & = -\left\langle v_{3,1},\cY^3_{22}((L_{-1}+3L_0)v_{2,1},1)v_{2,1}\right\rangle\nonumber\\
 & = -\left\langle v_{3,1},\left[(\mathrm{ad}\,L_0+2h_{2,1})\cY^3_{22}(v_{2,1},1)\right]v_{2,1}\right\rangle\nonumber\\
 & = -(h_{3,1}+h_{2,1}) =\frac{17}{4}.
\end{align*}
These show how to expand $\pi_3(\cY^3_{22}(v_{2,1},1)v_{2,1})$ in the basis dual to $\lbrace L_{-3}v_{3,1},L_{-1}L_{-2}v_{3,1}\rbrace$.

To determine the dual basis vectors, we can use the Virasoro algebra commutation relations to calculate
\begin{align*}
& \langle L_{-3} v_{3,1},L_{-3}v_{3,1}\rangle = 32, & \langle L_{-3} v_{3,1},L_{-1}L_{-2}v_{3,1}\rangle =2,\hspace{3em} & \\
 &\langle L_{-1}L_{-2}v_{3,1},L_{-3}v_{3,1}\rangle =2, & \langle L_{-1}L_{-2}v_{3,1},L_{-1}L_{-2}v_{3,1}\rangle =-55. &
\end{align*}
From these, it is easy to show that the dual basis is
\begin{equation*}
 [L_{-3}v_{3,1}]^* =\frac{55}{1764}L_{-3}v_{3,1}+\frac{1}{882} L_{-1}L_{-2}v_{3,1},\quad[L_{-1}L_{-2}v_{3,1}]^* =\frac{1}{882} L_{-3}v_{3,1}-\frac{8}{441}L_{-1}L_{-2}v_{3,1},
\end{equation*}
and thus
\begin{align*}
 \pi_3(\cY^3_{22}(v_{2,1},1)v_{2,1} & =-\frac{11}{2}\left(\frac{55}{1764}L_{-3}v_{3,1}+\frac{1}{882} L_{-1}L_{-2}v_{3,1}\right)\nonumber\\
 &\hspace{5em}+\frac{17}{4}\left(\frac{1}{882} L_{-3}v_{3,1}-\frac{8}{441}L_{-1}L_{-2}v_{3,1}\right)\nonumber\\
 &= -\frac{1}{6}L_{-3}v_{3,1}-\frac{1}{12} L_{-1}L_{-2}v_{3,1}.
\end{align*}
Now we can compute $c_3$ using the commutator formula for intertwining maps:
\begin{align*}
 c_3 & =\left\langle v_{2,1},\cY^2_{23}(v_{2,1},1)\pi_3(\cY^3_{22}(v_{2,1},1)v_{2,1})\right\rangle\nonumber\\
 & =-\frac{1}{12}\left\langle v_{2,1},\cY^2_{23}(v_{2,1},1)(2L_{-3}+L_{-1}L_{-2})v_{3,1}\right\rangle\nonumber\\
 & =\frac{1}{12}\left\langle v_{2,1},2\,\cY_{23}^2((L_{-1}-2L_0)v_{2,1},1)v_{3,1}+\cY^2_{23}(L_{-1}v_{2,1},1)L_{-2}v_{3,1} \right\rangle\nonumber\\
 & =\frac{1}{12}\left\langle v_{2,1}, 2\left[(\mathrm{ad}\,L_0-3h_{2,1})\cY^2_{23}(v_{2,1},1)\right]v_{3,1} +\left[(\mathrm{ad}\,L_0-h_{2,1})\cY^2_{23}(v_{2,1},1)\right]L_{-2}v_{3,1} \right\rangle\nonumber\\
 & =\frac{1}{6}(-2h_{2,1}-h_{3,1})\langle v_{2,1}\cY^2_{23}(v_{2,1},1)v_{3,1}\rangle-\frac{1}{12}(h_{3,1}+2)\left\langle v_{2,1},\cY^2_{23}(v_{2,1},1)L_{-2}v_{3,1}\right\rangle\nonumber\\
 & =\frac{11}{12}c_0-\frac{1}{12}\left\langle v_{2,1},\cY^2_{23}((L_{-1}-L_0)v_{2,1},1)v_{3,1}\right\rangle\nonumber\\
 & =\frac{11}{24}-\frac{1}{12}\left\langle v_{2,1},\left[(\mathrm{ad}\,L_0-2h_{2,1})\cY^2_{23}(v_{2,1},1)\right]v_{3,1}\right\rangle\nonumber\\
 & =\frac{11}{24}-\frac{1}{12}(-h_{2,1}-h_{3,1})c_0\nonumber\\
 & =\frac{11}{24}-\frac{17}{96}=\frac{9}{32}.
\end{align*}
Finally, we can compute $b$ and thus also the rigidity scalar $\mathfrak{R}$:
\begin{equation*}
 \mathfrak{R}=-b=\frac{25}{32}-\frac{9}{32}=\frac{1}{2}.
\end{equation*}
Thus $\cL_{2,1}$ is rigid and self-dual with evaluation $2\,p_1$ and coevaluation $i_1$.
\end{proof}

\section{Tensor products and rigidity}\label{sec:tens_prods_and_rigidity}

In this section we determine tensor products of irreducible modules in $\cO_{25}$, and we prove that $\cO_{25}$ is rigid. To prove the first theorem, we use the exact sequence \eqref{eqn:c25_tens_prod_exact} and the rigidity of $\cL_{2,1}$ proved in Theorem \ref{thm:L21_rigid}:
\begin{thm}
 The irreducible module $\cL_{r,1}$ is rigid in $\cO_{25}$ for all $r\in\ZZ_+$, and we have
 \begin{equation}\label{eqn:sl2_fusion}
  \cL_{r,1}\tens\cL_{r',1}\cong\bigoplus_{\substack{k=\vert r-r'\vert+1\\ k+r+r'\equiv 1\,(\mathrm{mod}\,2)}}^{r+r'-1} \cL_{k,1}
 \end{equation}
for $r,r'\in\ZZ_+$.
\end{thm}
\begin{proof}
 We first show that $\cL_{r,1}$ is rigid and that
 \begin{equation*}
  \cL_{2,1}\tens\cL_{r,1}\cong\left\lbrace\begin{array}{lll}
                                           \cL_{2,1} & \text{if} & r=1\\
                                           \cL_{r-1,1}\oplus\cL_{r+1,1} & \text{if} & r\geq 2\\
                                          \end{array}
\right. 
 \end{equation*}
by induction on $r$. The base case $r=1$ is clear because $\cL_{1,1}$ is the unit object of $\cO_{25}$, and we have proved the case $r=2$ in Theorems \ref{thm:L21_L21_tens} and \ref{thm:L21_rigid}. 

Now for some $r\geq 2$, assume we have proved that $\cL_{r,1}$ is rigid and that $\cL_{2,1}\tens\cL_{r,1}\cong\cL_{r-1,1}\oplus\cL_{r+1,1}$. Then $\cL_{2,1}\tens\cL_{r,1}$ is also rigid, since it is the tensor product of rigid modules, which means that $\cL_{r+1,1}$ is rigid, since it is a direct summand of a rigid module. Now consider the module $\cL_{2,1}\tens\cL_{r+1,1}$, which is rigid and self-dual because $\cL_{2,1}$ and $\cL_{r+1,1}$ are rigid and self-dual. Self-duality implies that the surjection $\cL_{2,1}\tens\cL_{r+1,1}\rightarrow\cL_{r+2,1}$ (guaranteed by \cite{OH}) dualizes to an injection $\cL_{r+2,1}\rightarrow\cL_{2,1}\tens\cL_{r+1,1}$. Thus the one-dimensional lowest conformal weight space $(\cL_{2,1}\tens\cL_{r+1,1})(0)=(\cL_{2,1}\tens\cL_{r+1,1})_{[h_{r+2,1}]}$ generates a submodule isomorphic to $\cL_{r+2,1}$, showing that in the exact sequence \eqref{eqn:c25_tens_prod_exact}, the submodule $\cJ_+=\ker\Pi_{r+1}^{+}$ is the maximal proper submodule of $\cV_{r+2,1}$.

We now have an exact sequence
\begin{equation*}
 0\longrightarrow\cL_{r+2,1}\longrightarrow\cL_{2,1}\tens\cL_{r+1,1}\longrightarrow\cV_{r,1}/\cJ_-\longrightarrow 0,
\end{equation*}
which splits because $\cL_{r+2,1}$ is both a submodule and quotient of $\cL_{2,1}\tens\cL_{r+1,1}$. Now self-duality of $\cL_{2,1}\tens\cL_{r+1,1}$ again implies that $\cL_{r,1}$ is a submodule as well as a quotient of $\cL_{2,1}\tens\cL_{r+1,1}$. As any singular vector generating the submodule $\cL_{r,1}$ must be contained in the summand $\cV_{r,1}/\cJ_-$, this shows that $\cJ_-$ is the maximal proper submodule of $\cV_{r,1}$ and therefore $\cL_{2,1}\tens\cL_{r+1,1}\cong\cL_{r,1}\oplus\cL_{r+2,1}$. This completes the inductive step.

So far, we have proved that $\cL_{r,1}$ is rigid for all $r\in\ZZ_+$, as well as the $r=2$ case of \eqref{eqn:sl2_fusion}. (The $r=1$ case is clear because $\cL_{1,1}$ is the unit object of $\cO_{25}$.) We now prove the remaining cases of \eqref{eqn:sl2_fusion} by induction on $r$. Assuming the formula for some $r\geq 2$ and all $r'\in\ZZ_+$,
\begin{align*}
 (\cL_{r-1,1}\tens\cL_{r',1})\oplus  ( & \cL_{r+1,1}\tens\cL_{r',1})  \cong (\cL_{2,1}\tens\cL_{r,1})\tens\cL_{r',1}\nonumber\\
 &\cong\cL_{2,1}\tens(\cL_{r,1}\tens\cL_{r',1})\cong\bigoplus_{\substack{k=\vert r-r'\vert+1\\ k+r+r'\equiv 1\,\mathrm{mod}\, 2)}}^{r+r'-1} (\cL_{2,1}\tens\cL_{k,1}).
\end{align*}
Since all tensor product modules here have finite length, the Krull-Schmidt Theorem guarantees that we can find the indecomposable summands of $\cL_{r+1,1}\tens\cL_{r',1}$ by subtracting the indecomposable (and irreducible) summands of $\cL_{r-1,1}\tens\cL_{r',1}$ from those on the right side. Analyzing the three possibilities $r>r'$, $r=r'$, and $r<r'$, we then find that
\begin{align*}
\cL_{r+1,1}\tens\cL_{r',1}\cong \bigoplus_{\substack{k=\vert r+1-r'\vert+1\\ k+r+r'\equiv 0\,(\mathrm{mod}\,2)}}^{r+r'} \cL_{k,1}
\end{align*}
as required.
\end{proof}

By the preceding theorem, all simple modules in $\cO_{25}$ are rigid. It then follows from \cite[Theorem 4.4.1]{CMY2} that the entire category $\cO_{25}$ of finite-length modules is rigid:
\begin{thm}
 The tensor category $\cO_{25}$ of $C_1$-cofinite grading-restricted generalized $V_{25}$-modules is rigid. Moreover, $\cO_{25}$ is a braided ribbon tensor category with natural twist isomorphism $e^{2\pi i L_0}$.
\end{thm}

\section{The semisimple subcategory \texorpdfstring{$\cO_{25}^0$}{O25-0}}\label{sec:O_25_0}

The category $\cO_{25}$ is not semisimple: besides containing arbitrary-length indecomposable quotients of Verma modules and their contragredients, it contains logarithmic self-extensions of all irreducible modules $\cL_{r,1}$ for $r\geq 2$ (see \cite[Section 5.4]{GK}). As a consequence, $\cO_{25}$ does not have enough projectives, and we expect that indecomposable objects in $\cO_{25}$ are essentially unclassifiable. However, the tensor product formula \eqref{eqn:sl2_fusion} shows that the irreducible modules of $\cO_{25}$ form the simple objects of a semisimple tensor subcategory, which we label $\cO_{25}^0$. Moreover, the fusion ring of $\cO_{25}^0$ is isomorphic to that of the category $\rep\mathfrak{sl}_2$ of finite-dimensional $\mathfrak{sl}_2$-modules under the map that sends $\cL_{r,1}$, $r\in\ZZ_+$, to the $r$-dimensional irreducible $\mathfrak{sl}_2$-module $V(r-1)$. In the rest of this section, we determine how $\cO_{25}^0$ and $\rep\mathfrak{sl}_2$ are related as braided tensor categories.

For $q \in \mathbb{C}^{\times}$, let $U_q(\mathfrak{sl}_2)$ be Lusztig's modified form of the Drinfeld-Jimbo quantum group specialized at $q$, and
 let $\cC(q, \mathfrak{sl}_2)$ be the category of finite-dimensional weight  $U_q(\mathfrak{sl}_2)$-modules. When $q = \pm 1$ or $q$ is not a root of unity, $\cC(q, \mathfrak{sl}_2)$ is semisimple with a fusion ring isomorphic to that of $\rep\mathfrak{sl}_2$ \cite{L}. Moreover, the tensor category $\cC(q, \mathfrak{sl}_2)$ is generated by a $2$-dimensional simple representation $L_q(1)$, called the standard representation, in the sense that every simple object of $\cC(q, \mathfrak{sl}_2)$ is a subquotient of a tensor power of $L_q(1)$.
 
 Conversely, if $\cC$ is a rigid semisimple tensor category with a simple self-dual generating object $X$, and if tensor products in $\cC$ agree with those of $\rep\mathfrak{sl}_2$ under the identification $X \mapsto V(1)$, then \cite[Theorem~$A_{\infty}$]{KW} says that there exists $q_{\cC}\in\CC$, either $q_{\cC} = \pm 1$ or $q_{\cC}$ is not a root of unity, and with $q_\cC^2$ uniquely determined up to inversion, such that $\cC$ is tensor equivalent to $\cC(q_{\cC}, \mathfrak{sl}_2)^\tau$ under a functor that sends $X$ to $L_{q_{\cC}}(1)$. Here $\tau$ denotes modification of the associativity isomorphisms in $\cC(q_{\cC}, \mathfrak{sl}_2)$ by a $3$-cocycle on $\ZZ/2\ZZ$. Up to coboundaries, there is only one non-trivial $3$-cocycle $\tau$ on $\ZZ/2\ZZ$: it changes the usual associativity isomorphism $L_{q_\cC}(1)\otimes (L_{q_\cC}(1)\otimes L_{q_\cC}(1)) \rightarrow (L_{q_\cC}(1)\otimes L_{q_\cC}(1))\otimes L_{q_\cC}(1)$ in $\cC(q_{\cC}, \mathfrak{sl}_2)$ by a sign. Note that the number denoted $q_\cC$ in \cite{KW} is actually the square of the number we have denoted $q_\cC$ here.
 
 To determine $q_{\cC}$, we look at the evaluation $e_X$ and the coevaluation $i_X$ such that the rigidity compositions
\begin{equation}\label{eq:rigidity1}
 X \xrightarrow{r_X^{-1}} X\tens{\bf 1} \xrightarrow{\Id_X\tens i_X} X \tens(X\tens X)\xrightarrow{\cA_{X,X,X}}(X \tens X)\tens X \xrightarrow{e_X\tens\Id_X} {\bf 1}\tens X \xrightarrow{l_X} X
\end{equation}
and
\begin{equation}\label{eq:rigidity2}
 X \xrightarrow{l_X^{-1}} {\bf 1}\tens X \xrightarrow{i_X \tens \Id_X} (X\tens X) \tens X \xrightarrow{\cA_{X,X,X}^{-1}} X\tens (X \tens X) \xrightarrow{\Id_X \tens e_X} X \tens {\bf 1} \xrightarrow{r_X} X
\end{equation}
equal $\Id_X$. The intrinsic dimension $d(X) = e_X \circ i_X \in \mathbb{C}$ is an invariant of the tensor category structure on $\cC$. So because the intrinsic dimension of $L_{q_{\cC}}(1)$ in $\cC(q_{\cC}, \mathfrak{sl}_2)$ is $-q_{\cC} - q_{\cC}^{-1}$ (see \cite[Exercise 8.18.8]{EGNO}), we can determine $q_{\cC}$ by comparing with $d(X)$ in $\cC$. Moreover, since $\pm q_{\cC}$ square to the same number, \cite[Theorem~$A_{\infty}$]{KW} implies that $\cC(q_{\cC}, \mathfrak{sl}_2)$ is tensor equivalent to a $3$-cocycle twist of $\cC(-q_{\cC}, \mathfrak{sl}_2)$, and this cocycle has to be non-trivial because the intrinsic dimensions of $L_{q_{\cC}}(1)$ and $L_{-q_{\cC}}(1)$ differ by a sign. Therefore $\cC(q_{\cC}, \mathfrak{sl}_2) \cong \cC(-q_{\cC}, \mathfrak{sl}_2)^{\tau}$ as tensor categories.

Now we turn to $\cO_{25}^0$, which is tensor generated by the simple self-dual object $\cL_{2,1}$. By the fusion rules \eqref{eqn:sl2_fusion} and \cite[Theorem~$A_{\infty}$]{KW}, the category $\cO_{25}^0$ is tensor equivalent to either $\cC(q, \mathfrak{sl}_2)$ or $\cC(q', \mathfrak{sl}_2)^{\tau}$ for some $q, q'\in\CC^\times$. The conclusion of the proof of Theorem \ref{thm:L21_rigid} shows that $d(\cL_{2,1}) = 2$. So if $\cO_{25}^0 \cong \cC(q, \mathfrak{sl}_2)$, we have
\[
d(L_q(1)) = -q - q^{-1} = 2,
\]
and thus $q= -1$, while if $\cO_{25}^0 \cong \cC(q', \mathfrak{sl}_2)^{\tau}$, a similar calculation shows $q' = 1$. Since the category $\cC(1, \mathfrak{sl}_2)^{\tau}$ is tensor equivalent to $\cC(-1, \mathfrak{sl}_2)$, we have:
\begin{prop}\label{prop: tens-equiv}
There are tensor equivalences $\cO_{25}^0 \cong \cC(-1, \mathfrak{sl}_2) \cong \cC(1, \mathfrak{sl}_2)^{\tau} \cong (\rep \mathfrak{sl}_2)^\tau$. 
\end{prop}

We next wish to relate the braidings on $\cO_{25}^0$ and $\cC(-1,\mathfrak{sl}_2)$. Thus we determine all possible braidings on $\cC(-1, \mathfrak{sl}_2)$ as well as the braiding on $\cO_{25}^0$, and then we find the unique braiding on $\cC(-1, \mathfrak{sl}_2)$ such that the tensor equivalence $\cC(-1, \mathfrak{sl}_2) \cong \cO_{25}^0$ of Proposition \ref{prop: tens-equiv} is a braided equivalence. Here we follow the discussion and methods of \cite[Sections 6 and 8]{GN}.

We first consider braidings on $\cC(q, \mathfrak{sl}_2)$. By naturality of the braiding isomorphisms and the hexagon axioms, a braiding $\cR$ is determined uniquely by its value $\cR_{X,X}$ on the generator $X = L_q(1)$. 
Set $$f_X : = i_X \circ e_X.$$
We claim that $f_X$ does not depend on the choice of $i_X$ and $e_X$: Since $X$ is simple, the spaces
\begin{equation*}
\mathrm{Hom}_{\cC(q,\mathfrak{sl}_2)}(X\otimes X,\vac)\cong\mathrm{End}_{\cC(q,\mathfrak{sl}_2)}(X) \cong\mathrm{Hom}_{\cC(q,\mathfrak{sl}_2)}(\vac, X\otimes X)
\end{equation*}
are one-dimensional. Thus any two choices of $e_X$ differ by a non-zero scalar multiple $c$, and the corresponding choices of $i_X$ differ by $c^{-1}$ in order to preserve \eqref{eq:rigidity1} and \eqref{eq:rigidity2}. So $f_X=i_X\circ e_X$ is well defined. 

Now since $\Id_{X\otimes X}$ and $f_X$ form a basis for the endomorphisms of $X\otimes X$, for any braiding $\cR$ on $\cC(q,\mathfrak{sl}_2)$, there exist $a, b \in \mathbb{C}$ such that
\[
\cR_{X,X} = a\cdot f_X + b \cdot \Id_{X\otimes X}.
\]
Conversely, if such an $\cR_{X,X}$ determines a braiding on $\cC(q,\mathfrak{sl}_2)$, then unital properties of the braiding, naturality, and the hexagon axioms imply the following:
\begin{align}\label{eq:braidings}
\cA_{X,X,X} & \circ(\Id_X \otimes i_X) \circ r_X^{-1}\nonumber\\
 & =\cA_{X,X,X}\circ (\Id_X\otimes i_X)\circ\cR_{\vac, X}\circ l_X^{-1}\nonumber\\
& =\cA_{X,X,X}\circ\cR_{X\otimes X,X}\circ(i_X\otimes\Id_X)\circ l_X^{-1}\nonumber\\
& =(\cR_{X,X}\otimes\Id_X)\circ\cA_{X,X,X}\circ(\Id_X\otimes\cR_{X,X})\circ\cA_{X,X,X}^{-1}\circ(i_X\otimes\Id_X)\circ l_X^{-1}.
\end{align}
As in \cite[Lemma~6.1]{GN}, this relation does not have too many solutions when $q=-1$:
\begin{lem}\label{lem:braidings}
When $q=-1$, there are precisely two solutions $\cR_{X, X}: X \otimes X \rightarrow X \otimes X$ to the equation \eqref{eq:braidings}, namely
\[
\cR_{X,X} = \pm i(f_X -\Id_{X\otimes X}).
\]
\end{lem}

\begin{proof}
Setting $\cR_{X,X}=a\cdot f_X+b\cdot\Id_{X\otimes X}$ and using the definition of $f_X$ and rigidity \eqref{eq:rigidity1} and \eqref{eq:rigidity2}, the right side of \eqref{eq:braidings} is
\begin{align*}
& a^2\cdot(f_X \otimes \Id_{X}) \circ \cA_{X,X,X} \circ (\Id_X \otimes f_{X}) \circ \cA_{X,X,X}^{-1} \circ (i_X \otimes \Id_X) \circ l_X^{-1}\\
&\quad \quad +ab\cdot\cA_{X,X,X}\circ(\Id_X\otimes f_X)\circ\cA_{X,X,X}^{-1}\circ(i_X\otimes\Id_X)\circ l_X^{-1}\nonumber\\
&\quad\quad + ab\cdot(f_{X} \otimes \Id_{X})  \circ (i_X \otimes \Id_X) \circ l_X^{-1}+b^2\cdot(i_X\otimes\Id_X)\circ l_X^{-1}\\
&\quad = a^2\cdot(i_X\otimes\Id_X)\circ(e_X\otimes\Id_X)\circ\cA_{X,X,X}\circ(\Id_X\otimes i_X)\circ r_X^{-1}\nonumber\\
&\quad\quad +ab\cdot\cA_{X,X,X}\circ(\Id_X\otimes i_X)\circ r_X^{-1}+(d(X)ab+b^2)\cdot(i_X\otimes\Id_X)\circ l_X^{-1}\nonumber\\
& \quad=(a^2+d(X)ab+b^2)\cdot(i_X\otimes\Id_X)\circ l_X^{-1}+ab\cdot\cA_{X,X,X}\circ(\Id_X\otimes i_X)\circ r_X^{-1}.
\end{align*}
Recall that $d(X)=2$ when $q=-1$, and also note that $(i_X\otimes\Id_X)\circ l_X^{-1}$ and $\cA_{X,X,X}\circ(\Id_X\otimes i_X)\circ r_X^{-1}$ are linearly independent; this can be seen by composing both morphisms with
\begin{equation*}
d(X)\cdot l_X\circ(e_X\otimes\Id_X)-r_X\circ(\Id_X\otimes e_X)\circ\cA_{X,X,X}^{-1}.
\end{equation*}
Thus $\cR_{X,X}=a\cdot f_X+b\cdot\Id_{X\otimes X}$ solves \eqref{eq:braidings} if and only if
\begin{equation*}
b=-a\quad\text{and}\quad b=a^{-1}.
\end{equation*}
That is, $a=\pm i$ and $b=\mp i$.
\end{proof}

With $q=-1$, it is easy to check that 
\[
(f_X -  \Id_{X\otimes X})^2 = (d(X) - 2)\circ f_X + \Id_{X\otimes X} = \Id_{X\otimes X}.
\]
Thus the two possible braidings $\cR_{X,X}$ in Lemma \ref{lem:braidings} are mutual inverses. Since the tensor category $\cC(-1, \mathfrak{sl}_2)$ is indeed braided, both possibilities for $\cR_{X,X}$ extend to braidings on $\cC(-1,\mathfrak{sl}_2)$, and thus $\cC(-1,\mathfrak{sl}_2)$ admits exactly two braidings. We need to determine which of these two braidings makes the tensor equivalence $\cO_{25}^0 \cong \cC(-1, \mathfrak{sl}_2)$ of Proposition \ref{prop: tens-equiv} a braided equivalence, so we determine the braiding on $\cO_{25}^0$ next.

Just as in Lemma \ref{lem:braidings}, there are exactly two braidings on $\cO_{25}^0$, completely determined by
\begin{equation}\label{eq:braiding virasoro}
\cR_{\cL_{2,1}, \cL_{2,1}} = i(f_{\cL_{2,1}} - \Id_{\cL_{2,1}\tens\cL_{2,1}})\;\;\; {\rm or}\;\;\; \cR_{\cL_{2,1}, \cL_{2,1}} = -i(f_{\cL_{2,1}} - \Id_{\cL_{2,1}\tens\cL_{2,1}}).
\end{equation}
One of them is the braiding specified in \cite{HLZ8}, and the other is its inverse; we would like to determine which is the braiding of \cite{HLZ8}. In fact, the construction in \cite{HLZ8} (see also \cite[Section 3.3.4]{CKM1}) shows that $\cR_{\cL_{2,1},\cL_{2,1}}$ satisfies
\begin{equation}\label{eq:skew}
\cR_{\cL_{2,1},\cL_{2,1}}(\cY_\tens(v_{2,1}, x)v_{2,1}) = e^{xL(-1)}\cY_\tens(v_{2,1}, e^{\pi i}x)v_{2,1},
\end{equation}
where $v_{2,1}\in\cL_{2,1}$ is a generating vector of conformal weight $h_{2,1}=-\frac{5}{4}$ and $\cY_\tens$ is the tensor product intertwining operator of type $\binom{\cL_{2,1}\tens\cL_{2,1}}{\cL_{2,1}\,\cL_{2,1}}$. 

The method of \cite[Sections~7 and 8]{GN} shows how to determine which possibility in \eqref{eq:braiding virasoro} corresponds to \eqref{eq:skew}, by comparing $e_{\cL_{2,1}} \circ \cR_{\cL_{2,1},\cL_{2,1}}$ with $e_{\cL_{2,1}}$. Similar to equation \eqref{phi2} in the proof of Theorem \ref{thm:L21_rigid}, we can choose $e_{\cL_{2,1}}$ so that the intertwining operator $\cE:=e_{\cL_{2,1}}\circ\cY_\tens$ of type $\binom{\cL_{1,1}}{\cL_{2,1}\,\cL_{2,1}}$ satisfies
\begin{equation}\label{eqn:eval_intw_op}
\cE(v_{2,1},x)v_{2,1}\in x^{-2h_{2,1}}(\vac +x\cL_{1,1}[[x]])=x^{5/2}(\vac+x\cL_{1,1}[[x]]).
\end{equation}
So by \eqref{eq:skew},
\begin{align*}
(e_{\cL_{2,1}}\circ\cR_{\cL_{2,1},\cL_{2,1}}\circ\cY_\tens)(v_{2,1},x)v_{2,1} & = e^{xL(-1)}\cE(v_{2,1},e^{\pi i}x)v_{2,1}\nonumber\\
& \in(e^{\pi i} x)^{5/2}(\vac+x\cL_{1,1}[[x]]) = x^{5/2}(i\cdot\vac+x\cL_{1,1}[[x]]).
\end{align*}
Since $\cE$ spans the space of intertwining operators of type $\binom{\cL_{1,1}}{\cL_{2,1}\,\cL_{2,1}}$, it follows that
\begin{equation}\label{eqn:braid_and_eval}
e_{\cL_{2,1}}\circ\cR_{\cL_{2,1},\cL_{2,1}}\circ\cY_\tens=i\cdot\cE,
\end{equation}
equivalently $e_{\cL_{2,1}} \circ \cR_{\cL_{2,1}, \cL_{2,1}} = i \cdot e_{\cL_{2,1}}$. On the other hand,
\[
e_{\cL_{2,1}} \circ \pm i(f_{\cL_{2,1}} - \Id_{\cL_{2,1}\tens\cL_{2,1}}) = \pm i(d(\cL_{2,1})-1)e_{\cL_{2,1}} =\pm i\cdot e_{\cL_{2,1}}.
\]
As a result, the braiding on $\cO_{25}^0$ given by \cite{HLZ8} satisfies
\begin{equation}\label{eq:correct braiding virasoro}
\cR_{\cL_{2,1}, \cL_{2,1}} = i(f_{\cL_{2,1}} - \Id_{\cL_{2,1}\tens\cL_{2,1}}).
\end{equation}
To conclude,
\begin{thm}\label{prop:braiding1}
The tensor category $\cC(-1, \mathfrak{sl}_2)$ equipped with the braiding determined by $\cR_{L_{-1}(1), L_{-1}(1)} = i(f_{L_{-1}(1)} -  \Id_{L_{-1}(1)\otimes L_{-1}(1)})$ is braided tensor equivalent to $\cO_{25}^0$.
\end{thm}
\begin{proof}
Proposition \ref{prop: tens-equiv} gives a tensor equivalence $\cF: \cO_{25}^0\rightarrow\cC(-1, \mathfrak{sl}_2)$ such that $\cF(\cL_{2,1})=L_{-1}(1)$. Moreover, $\cF$ becomes a braided tensor equivalence when $\cC(-1,\mathfrak{sl}_2)$ is equipped with the braiding induced by the braiding on $\cO_{25}^0$ via $\cF$. This is the braiding $\cR$ on $\cC(-1,\mathfrak{sl}_2)$ uniquely determined by the commutativity of the diagram
\begin{equation*}
\xymatrixcolsep{5pc}
\xymatrix{
L_{-1}(1)\otimes L_{-1}(1) \ar[r]^{\cR_{L_{-1}(1),L_{-1}(1)}} \ar[d]^F & L_{-1}(1)\otimes L_{-1}(1) \ar[d]^F\\
\cF(\cL_{2,1}\tens\cL_{2,1}) \ar[r]^{\cF(\cR_{\cL_{2,1},\cL_{2,1}})} & \cF(\cL_{2,1}\tens\cL_{2,1})\\
}
\end{equation*}
where $F: L_{-1}(1)\otimes L_{-1}(1)=\cF(\cL_{2,1})\otimes\cF(\cL_{2,1})\rightarrow\cF(\cL_{2,1}\tens\cL_{2,1})$ is the natural isomorphism that is part of the data of the tensor functor $\cF$. So using \eqref{eq:correct braiding virasoro}, 
\begin{align*}
\cR_{L_{-1}(1),L_{-1}(1)} & = F^{-1}\circ i\cF( f_{\cL_{2,1}}-\Id_{\cL_{2,1}\tens\cL_{2,1}})\circ F =i(F^{-1}\circ \cF(f_{\cL_{2,1}})\circ F-\Id_{L_{-1}(1)\otimes L_{-1}(1)}).
\end{align*}
It remains to show that $F^{-1}\circ\cF(f_{\cL_{2,1}})\circ F=f_{L_{-1}(1)}$. Indeed, properties of tensor functors (see for example \cite[Exercise 2.10.6]{EGNO} or \cite[Proposition 2.77]{CKM1}) show we can take 
\begin{equation*}
e_{L_{-1}(1)}=\varphi\circ\cF(e_{\cL_{2,1}})\circ F\quad\text{and}\quad i_{L_{-1}(1)}=F^{-1}\circ\cF(i_{\cL_{2,1}})\circ \varphi^{-1},
\end{equation*}
where $\varphi: \cF(\cL_{1,1})\rightarrow \vac$ is the isomorphism between units that is part of the data of the tensor functor $\cF$. Thus $F^{-1}\circ\cF(f_{\cL_{2,1}})\circ F= i_{L_{-1}(1)} \circ e_{L_{-1}(1)} = f_{L_{-1}(1)}$, as desired. 
\end{proof}

In the next sections, we will also work with the category $\cO_1$ of $C_1$-cofinite grading-restricted generalized modules for the (simple) $c=1$ Virasoro vertex operator algebra $V_1$. By \cite{CJORY}, $\cO_1$ equals the category of finite-length $\cV ir$-modules of central charge $1$ whose composition factors are irreducible quotients of reducible Verma modules, and $\cO_1$ admits the vertex algebraic braided tensor category structure of \cite{HLZ1}-\cite{HLZ8}. By \cite[Example 4.12]{McR} and \cite[Remark 4.4.6]{CMY2}, $\cO_1$ is also rigid. As at $c=25$, $\lbrace\cL_{r,1}\,\vert\,r\in\ZZ_+\rbrace$ are all the simple objects of $\cO_1$, up to isomorphism; at $c=1$, $\cL_{r,1}$ is the irreducible $\cV ir$-module with lowest conformal weight $\frac{1}{4}(r-1)^2$ (by the $t=1$ case of \eqref{eqn:h_rs_def}).

Although $\cO_1$ is not semisimple, \cite[Example 4.12]{McR} shows that the $V_1$-modules $\cL_{r,1}$, $r\in\ZZ_+$, form the simple objects of a semisimple tensor subcategory $\cO_1^0$ that is tensor equivalent to $\cC(1,\mathfrak{sl}_2)^\tau\cong\cC(-1,\mathfrak{sl}_2)$. The braiding on $\cO_1^0$ was determined in \cite{McR}, but we can also use the method of this section:

\begin{prop}\label{prop:braiding2}
The tensor category $\cC(-1, \mathfrak{sl}_2)$ equipped with the braiding determined by $\cR_{L_{-1}(1), L_{-1}(1)} = -i(f_{L_{-1}(1)} -  \Id_{L_{-1}(1)\otimes L_{-1}(1)})$ is braided tensor equivalent to $\cO_{1}^0$. In particular, there is a braid-reversed tensor equivalence $\cO_1^0\rightarrow\cO_{25}^0$ that identifies the simple modules $\cL_{r,1}$ in both categories.
\end{prop}
\begin{proof}
The proof is the same as for $\cO_{25}^0$, except that the lowest power of $x$ in \eqref{eqn:eval_intw_op} is $x^{-1/2}$ since now $h_{2,1}=\frac{1}{4}$. With this change, \eqref{eqn:braid_and_eval} becomes
\begin{equation*}
e_{\cL_{2,1}}\circ\cR_{\cL_{2,1},\cL_{2,1}}\circ\cY_\tens=e^{-\pi i/2}\cdot\cE,
\end{equation*}
and then \eqref{eq:correct braiding virasoro} changes to
\begin{equation}\label{eq:correct braiding virasoro_1}
\cR_{\cL_{2,1}, \cL_{2,1}} = -i(f_{\cL_{2,1}} - \Id_{\cL_{2,1}\tens\cL_{2,1}}).
\end{equation}
The braided tensor equivalence of $\cO_1^0$ with $\cC(-1,\mathfrak{sl}_2)$ (equipped with the indicated braiding) follows exactly as in the proof of Theorem \ref{prop:braiding1}, and then the braid-reversed equivalence with $\cO_{25}^0$ is a corollary of Theorem \ref{prop:braiding1}.
\end{proof}

\section{The Virasoro vertex operator algebra \texorpdfstring{$V_{25}$}{V(25)} as a \texorpdfstring{$PSL_2(\CC)$}{PSL(2,C)}-orbifold}\label{sec:W(-1)}

In this section, we construct and explore the representation theory of a vertex algebra $\cW(-1)$ that contains the $c=25$ Virasoro vertex operator algebra $V_{25}$ as a $PSL_2(\CC)$-orbifold, that is, as the fixed-point subalgebra of an action of $PSL_2(\CC)$ by automorphisms. The algebra $\cW(-1)$ is a conformal vertex algebra in the sense of \cite{HLZ1}, since it has a conformal vector but infinite-dimensional conformal weight spaces and no lower bound on conformal weights. For motivation, recall that for $p\in\ZZ_+$, the Virasoro vertex operator algebra $V_{13-6p-6p^{-1}}$ is the $PSL_2(\CC)$-orbifold of a simple $C_2$-cofinite vertex operator algebra $\cW(p)$: for $p=1$, this is the $\mathfrak{sl}_2$-root lattice vertex operator algebra \cite{DG, Mi}, while for $p\geq 2$, this is the triplet $W$-algebra \cite{ALM}. Thus we are now constructing the analogous vertex algebra $\cW(-1)$ for $p=-1$, equivalently, $c=25$.

The idea for obtaining $\cW(-1)$ is to use the tensor equivalence $\cO_1^0\cong\cO_{25}^0$ of the previous section to transfer the vertex operator algebra structure on the simple affine vertex operator algebra $L_1(\mathfrak{sl}_2)$ (which is an extension of $V_1$) to a vertex operator structure on an extension of $V_{25}$. To do so, we need some general results on commutative algebras in braided tensor categories, which we place Appendix \ref{app:alg_in_tens_cat}. See also for example \cite{KO, CKM1} for more notation and results related to commutative algebras in braided tensor categories.

 We now take the Ind-categories, or direct limit completions, $\ind(\cO_{1}^0)$ and $\ind(\cO_{25}^0)$, as defined in \cite{CMY1}. Specifically, denoting the simple objects of $\cO_i^0$ for $i=1,25$ by $\cL_{r,1}^{(i)}$, $\ind(\cO_i^0)$ is the semisimple braided tensor category of $V_i$-modules which are isomorphic to possibly infinite direct sums of modules $\cL_{r,1}^{(i)}$, $r\in\ZZ_+$. The braid-reversed tensor equivalence $\cO_{1}^0\rightarrow\cO_{25}^0$ of Proposition \ref{prop:braiding2} extends to a braid-reversed tensor equivalence
\begin{equation*}
\cF: \ind(\cO_{1}^0)\longrightarrow\ind(\cO_{25}^0)
\end{equation*}
such that $\cF(\cL_{r,1}^{(1)})\cong\cL_{r,1}^{(25)}$ for all $r\in\ZZ_+$. 

\begin{thm}\label{thm:W_minus_one}
There is a unique (up to isomorphism) conformal vertex algebra extension $\cW(-1)$ of $V_{25}$ such that
\begin{equation}\label{eqn:W_minus_1_decomp}
\cW(-1)\cong\bigoplus_{n= 0}^\infty V(2n)\otimes\cL_{2n+1,1}^{(25)}
\end{equation}
as $V_{25}$-modules, where $V(2n)$ denotes the irreducible $(2n+1)$-dimensional $\mathfrak{sl}_2$-module. Moreover, $PSL_2(\CC)$ acts by conformal vertex algebra automorphisms on $\cW(-1)$ with fixed-point subalgebra $V_{25}$ such that \eqref{eqn:W_minus_1_decomp} also gives the decomposition of $\cW(-1)$ as a $PSL_2(\CC)$-module.
\end{thm}
\begin{proof}
For the existence of $\cW(-1)$, recall that the simple affine vertex operator algebra $L_1(\mathfrak{sl}_2)$ of $\mathfrak{sl}_2$ at level $1$ (equivalently, the $\mathfrak{sl}_2$-root lattice vertex operator algebra) has automorphism group $PSL_2(\CC)$ with fixed-point subalgebra $V_1$ such that
\begin{equation*}
L_1(\mathfrak{sl}_2)\cong\bigoplus_{n= 0}^\infty V(2n)\otimes\cL_{2n+1,1}^{(1)}
\end{equation*}
as $PSL_2(\CC)\times V_1$-modules \cite{DG, Mi}. By \cite[Theorem 3.2]{HKL} or \cite[Theorem 7.5]{CMY1}, $L_1(\mathfrak{sl}_2)$ is a commutative $\ind(\cO_1^0)$-algebra, so $\cW(-1)=\cF(L_1(\mathfrak{sl}_2)$ is a simple commutative $\ind(\cO_{25}^0)$-algebra by  Corollary \ref{cor:F_preserve_simple} in the appendix. Thus again by \cite[Theorem 7.5]{CMY1}, $\cW(-1)$ is a simple conformal vertex algebra extension of $V_{25}$ such that $\cW(-1)\cong\bigoplus_{n= 0}^\infty V(2n)\otimes\cL_{2n+1,1}^{(25)}$ as a $V_{25}$-module. 

Moreover, if $g$ is an automorphism of $L_1(\mathfrak{sl}_2)$, then $\cF(g)$ is an automorphism of $\cW(-1)$. Thus $PSL_2(\CC)$ acts on $\cW(-1)$ by conformal vertex algebra automorphisms; since each automorphism of $\cW(-1)$ is in particular a $V_{25}$-module automorphism, we get a decomposition
\begin{equation*}
\cW(-1)\cong\bigoplus_{n = 0}^\infty \cF(V(2n)\otimes\cL_{2n+1,1}^{(1)})
\end{equation*}
as a $PSL_2(\CC)$-module, where each $g\in PSL_2(\CC)$ acts on $\cF(V(2n)\otimes\cL_{2n+1,1}^{(1)})$ by $\cF(g\otimes\Id)$. Let $T\subseteq PSL_2(\CC)$ denote a maximal torus and let $v_i\in V(2n)$ for $i=-n,-n+1,\ldots,n-1,n$ denote weight vectors of distinct weights $\lambda_{2i}: T\rightarrow\CC^\times$. Then for each $i$ and each $t\in T$, we have a commutative diagram
\begin{equation*}
\xymatrixcolsep{6pc}
\xymatrix{
\cF(v_i\otimes\cL_{2n+1,1}^{(1)}) \ar[d]^{\lambda_{2i}(t)\cdot\Id} \ar[r]^(.47){\cF(q_i\otimes\Id)} & \cF(V(2n)\otimes\cL_{2n+1,1}^{(1)}) \ar[d]^t\\
\cF(v_i\otimes\cL_{2n+1,1}^{(1)}) \ar[r]^(.47){\cF(q_i\otimes\Id)} & \cF(V(2n)\otimes\cL_{2n+1,1}^{(1)}) \\
}
\end{equation*}
where $q_i: \CC v_i\rightarrow V(2n)$ is the $T$-module injection. Note that since $q_i\otimes\Id$ is non-zero and $\cF$ is faithful, $\cF(q_i\otimes\Id)$ is also non-zero and thus injective (since its domain is an irreducible $V_{25}$-module). This shows that the lowest conformal weight space of $\cF(V(2n)\otimes\cL_{2n+1,1}^{(1)})$, which is a $(2n+1)$-dimensional $PSL_2(\CC)$-module, has the same character as $V(2n)$, and thus is isomorphic to $V(2n)$. It follows that
\begin{equation*}
\cF(V(2n)\otimes\cL_{2n+1,1}^{(1)})\cong V(2n)\otimes\cL_{2n+1,1}^{(25)}
\end{equation*}
as a $PSL_2(\CC)\times V_{25}$-module, proving that \eqref{eqn:W_minus_1_decomp} gives the decomposition of $\cW(-1)$ as a $PSL_2(\CC)\times V_{25}$-module.

For the uniqueness assertion in the theorem, let $V$ be any simple conformal vertex algebra extension of $V_{25}$ such that $V\cong\bigoplus_{n = 0}^\infty V(2n)\otimes\cL_{2n+1,1}^{(25)}$ as a $\cV ir$-module. Then $V\cong\cF(A)$ for some commutative $\ind(\cO_1^0)$-algebra $A$ (this can be proved similarly to Proposition \ref{prop:F_equiv_on_Rep_A} in Appendix \ref{app:alg_in_tens_cat}), and $A$ is simple by Corollary \ref{cor:F_preserve_simple}. If we can show $A\cong L_1(\mathfrak{sl}_2)$, then we have $V\cong\cW(-1)$. Thus it is sufficient to prove uniqueness of the simple vertex operator algebra structure on $\bigoplus_{n = 0}^\infty V(2n)\otimes\cL_{2n+1,1}^{(1)}$ extending $V_1$. This is proved in Theorem \ref{thm:L1_sl_2_unique} in Appendix \ref{app:L1_sl_2_unique}.
\end{proof}

\begin{rem}
We can also decompose $\cW(-1)$ as a $PSL_2(\CC)$-module using $\mathfrak{sl}_2$. First, $\mathfrak{sl}_2$ acts on $L_1(\mathfrak{sl}_2)$ by vertex operator algebra derivations (equivalently, $\ind(\cO_1^0)$-algebra derivations) and thus by the tensor equivalence $\cF$ acts on $\cW(-1)$ by $\ind(\cO_{25}^0)$-algebra derivations (equivalently, conformal vertex algebra derivations). Then just as in the proof of Theorem \ref{thm:W_minus_one}, the lowest conformal weight space of each $\cF(V(2n)\otimes\cL_{2n+1,1}^{(1)})$ is a $(2n+1)$-dimensional $\mathfrak{sl}_2$-module with the same character as $V(2n)$ and thus is isomorphic to $V(2n)$. It follows that \eqref{eqn:W_minus_1_decomp} gives the decomposition of $\cW(-1)$ as an $\mathfrak{sl}_2\times V_{25}$-module. Then since the action of $\mathfrak{sl}_2$ is locally finite, it integrates to an action of $SL_2(\CC)$ on $\cW(-1)$ by conformal vertex algebra automorphisms which by the decomposition \eqref{eqn:W_minus_1_decomp} descends to a faithful action of $PSL_2(\CC)$.
\end{rem}

\begin{rem}
We expect that the Virasoro vertex operator algebras $V_{13+6p+6p^{-1}}$ for $p\in\ZZ_{\geq 2}$ can also be realized as $PSL_2(\CC)$-orbifolds of simple conformal vertex algebras $\cW(-p)$. It would be interesting to see whether all such vertex algebras $\cW(-p)$ for $p\geq 1$ are $C_2$-cofinite, although in contrast to the $\NN$-graded case, the implications of $C_2$-cofiniteness for general $\ZZ$-graded conformal vertex algebras are not completely understood.
\end{rem}

We can now determine the representation theory of the conformal vertex algebra $\cW(-1)$ using Corollary \ref{cor:rep_0_A_rep_0_F(A)} in Appendix \ref{app:alg_in_tens_cat} and the well-known representation theory of $L_1(\mathfrak{sl}_2)$. In particular, viewing $\cW(-1)$ as a commutative algebra in $\ind(\cO_{25}^0)$, we have its braided tensor category $\rep^0\cW(-1)$ of local modules. As in \cite[Theorem 3.4]{HKL}, $\rep^0\cW(-1)$ is just the category of modules for $\cW(-1)$ (considered as a conformal vertex algebra) which are objects of $\ind(\cO_{25}^0)$ when viewed as $V_{25}$-modules. Since $L_1(\mathfrak{sl}_2)$ is a regular vertex operator algebra with two distinct simple modules, Corollary \ref{cor:rep_0_A_rep_0_F(A)} implies that $\rep^0\cW(-1)$ is semisimple with two distinct simple objects: $\cW(-1)$ itself and another simple $\cW(-1)$-module $\mathcal{X}$ such that
\begin{equation*}
\cX\cong\bigoplus_{n = 1}^\infty V(2n-1)\otimes\cL_{2n,1}^{(25)}
\end{equation*}
as a $V_{25}$-module. The only non-trivial tensor product in $\rep^0\cW(-1)$ is $\cX\tens\cX\cong\cW(-1)$, and moreover:
\begin{thm}\label{thm:W(-1)_rep_theory}
The category of finite-length $\cW(-1)$-modules which restrict to $V_{25}$-modules in $\ind(\cO_{25}^0)$ is a braided tensor category which is braid-reversed tensor equivalent to the modular tensor category of $L_1(\mathfrak{sl}_2)$-modules.
\end{thm}
\begin{rem}
The above discussion and theorem suggest that $\cW(-1)$ may be a regular conformal vertex algebra, in the sense that all its weak modules are direct sums of $\cW(-1)$ and $\cX$. However, our methods here do not allow us to rule out $\cW(-1)$-modules which are not unions of their $C_1$-cofinite $V_{25}$-submodules. Also, it is unlikely that Zhu algebra methods would help with the classification of simple $\cW(-1)$-modules, since $\cW(-1)$-modules need not be $\NN$-gradable.
\end{rem}

We can also view $\cW(-1)$ as a commutative algebra in the larger braided tensor category $\ind(\cO_{25})$, and we then have the tensor category $\rep\cW(-1)$ of not-necessarily-local $\cW(-1)$-modules in $\ind(\cO_{25})$. We now show that $\cW(-1)$ and $\cX$ are the only simple objects of $\rep\cW(-1)$. To do so, recall from \cite{KO, CKM1} that there is a tensor functor of induction
\begin{align*}
\cF_{\cW(-1)}: \cO_{25} & \rightarrow\rep\cW(-1)\nonumber\\
 W &\mapsto \cW(-1)\tens W\nonumber\\
 f & \mapsto \Id_{\cW(-1)}\tens f
\end{align*}
The induction functor satisfies Frobenius reciprocity, in the sense that for any objects $W$ of $\cO_{25}$ and $X$ of $\rep\cW(-1)$, there is a natural isomorphism
\begin{equation*}
\hom_{V_{25}}(W,X)\xrightarrow{\cong}\hom_{\cW(-1)}(\cF_{\cW(-1)}(W),X).
\end{equation*}
The arguments in the proofs of \cite[Lemma 7.3 and Proposition 7.4]{MY} (which use Frobenius reciprocity) show that:
\begin{prop}\label{prop:W_minus_1_inductions}
For $r\in\ZZ_+$,
\begin{equation*}
\cF_{\cW(-1)}(\cL_{r,1})\cong\left\lbrace\begin{array}{ll}
\cW(-1)^{\oplus r} & \text{if $r$ is odd}\\
\cX^{\oplus r} &\text{if $r$ is even}\\
\end{array}
\right. .
\end{equation*}
\end{prop}

As a consequence:
\begin{prop}
Every simple object of $\rep\cW(-1)$ is isomorphic to either $\cW(-1)$ or $\cX$.
\end{prop}
\begin{proof}
Let $X$ be any simple object of $\rep\cW(-1)$. Since $X$ restricts to a $V_{25}$-module in $\cO_{25}$, and since every object of $\cO_{25}$ has finite length, $X$ contains a simple $V_{25}$-submodule $\cL_{r,1}$ for some $r\in\ZZ_+$. Then by Frobenius reciprocity, the $V_{25}$-module inclusion $\cL_{r,1}\hookrightarrow X$ induces a non-zero $\cW(-1)$-module homomorphism $\cF_{\cW(-1)}(\cL_{r,1})\rightarrow X$. Since $X$ is simple, this non-zero map is surjective, so that by Proposition \ref{prop:W_minus_1_inductions}, $X$ is a simple quotient of either $\cW(-1)^{\oplus r}$ or $\cX^{\oplus r}$. Thus $X$ is isomorphic to either $\cW(-1)$ or $\cX$.
\end{proof}

\begin{rem}
We can use the methods of this section to construct another simple conformal vertex algebra extension of $V_{25}$ that is analogous to the rank-one Heisenberg subalgebra $\cM(1)\subseteq L_1(\mathfrak{sl}_2)$. Again using the braid-reversed tensor equivalence $\cF:\ind(\cO_1^0)\rightarrow\ind(\cO_{25}^0)$, we define $\cM(-1):=\cF(\cM(1))$, a simple conformal vertex algebra isomorphic to $\bigoplus_{n=0}^\infty \cL_{2n+1,1}$ as a $V_{25}$-module. Alternatively, we could define $\cM(-1)$ to be the $U(1)$-orbifold of $\cW(-1)$. We can then consider the category of finite-length $\cM(-1)$-modules which restrict to $V_{25}$-modules in $\ind(\cO_{25}^0)$. Using Corollary \ref{cor:rep_0_A_rep_0_F(A)}, this category of $\cM(-1)$-modules is a rigid braided tensor category that is braid-reversed tensor equivalent to the category of finite-length $\cM(1)$-modules which restrict to $V_1$-modules in $\ind(\cO_1^0)$. In particular, the simple $\cM(-1)$-modules in this category are simple currents parametrized by continuous characters of $U(1)$. 
\end{rem}

\section{The chiral universal centralizer of \texorpdfstring{$SL_2$}{SL(2)} at level \texorpdfstring{$-1$}{-1}}\label{sec:chiral_univ_center}

In this section, we determine the representation theory of a simple conformal vertex algebra extension of $V_1\otimes V_{25}$. This algebra was constructed in \cite[Section 7]{Ar}, where it was called the \textit{chiral universal centralizer} algebra of $SL_2$ at level $-1$ and denoted $\mathbf{I}^{-1}_{SL_2}$. For a general level $k=-2+t$, the chiral universal centralizer $\mathbf{I}^k_{SL_2}$ is an extension of the tensor product of Virasoro vertex operator algebras at central charge $13\pm 6t\pm 6t^{-1}$ and can be obtained from the vertex algebra of chiral differential operators on $SL_2$ at level $k$ by a two-step process of quantum Drinfeld-Sokolov reduction. In the case $t\notin\QQ$, this construction first appeared in \cite{FS}, and an alternative explicit construction was given in \cite{FZ2}, where $\mathbf{I}^k_{SL_2}$ was called a \textit{modified regular representation} of the Virasoro algebra.

We take $k=-1,-3$. For this pair of levels, $\mathbf{I}_{SL_2}^{-1}$ is a simple conformal vertex algebra extension of $V_1\otimes V_{25}$ which is semisimple as a $V_1\otimes V_{25}$-module; see \cite[Section 5]{McR-cosets} for a proof which applies to all level pairs $-2\pm\frac{1}{p}$, $p\in\mathbb{Z}_+$, based on results from \cite{Ar} and the semisimplicity of the Kazhdan-Lusztig categories of modules for affine $\mathfrak{sl}_2$ at these levels. More specifically,
\begin{equation}\label{eqn:gluing_structure}
\mathbf{I}^{-1}_{SL_2}\cong\bigoplus_{r\in\ZZ_+} \cL_{r,1}^{(1)}\otimes\cL_{r,1}^{(25)}
\end{equation}
as a $V_1\otimes V_{25}$-module, where we use $\cL_{r,1}^{(i)}$ for $i=1,25$ to denote the simple objects of $\cO_i^0$. For $r\in\ZZ_+$, the simple $V_1\otimes V_{25}$-module $\cL_{r,1}^{(1)}\otimes\cL_{r,1}^{(25)}$ has lowest conformal weight
\begin{equation*}
\frac{1}{4}(r-1)^2-\frac{1}{4}(r+1)^2+1=-r+1\in\ZZ
\end{equation*}
by the $t=\pm 1$ cases of \eqref{eqn:h_rs_def}. Thus $\mathbf{I}^{-1}_{SL_2}$ is $\ZZ$-graded, but there is no lower bound on the conformal weights. Another way to obtain a simple conformal vertex algebra structure on \eqref{eqn:gluing_structure} uses the method of gluing vertex algebras from \cite{CKM2}. Namely, because $\cO_1^0$ and $\cO_{25}^0$ are braid-reversed equivalent rigid semisimple tensor categories by Proposition \ref{prop:braiding2}, \eqref{eqn:gluing_structure} is the canonical algebra in the Ind-category of the Deligne tensor product $\cO_1^0\otimes\cO_{25}^0$, and thus is a simple conformal vertex algebra by \cite[Theorem 7.5]{CMY1} (here we use $\otimes$ to denote the Deligne tensor product of braided tensor categories to avoid confusion with our notation $\tens$ for vertex algebraic tensor products). In particular, the existence assertion of the next theorem follows from either \cite[Section 7]{Ar} or \cite[Main Theorem 3(1)]{CKM2}, and we prove the uniqueness assertion in Appendix \ref{app:chiral_univ_cent_unique}:

\begin{thm}\label{thm:glue}
There is a unique (up to isomorphism) simple conformal vertex algebra extension of $V_1\otimes V_{25}$ with the decomposition \eqref{eqn:gluing_structure} as a $V_1\otimes V_{25}$-module.
\end{thm}

We now use the theory of vertex algebra extensions developed in \cite{HKL, CKM1, CKM2, CMY1} to describe the representations of $\mathbf{I}_{SL_2}^{-1}$ in $\ind(\cO_1^0\otimes\cO_{25}^0)$. For notational simplicity, we set $A=\mathbf{I}^{-1}_{SL_2}$. Let $\rep A$ denote the category of (possibly non-local) $A$-modules that restrict to $V_1\otimes V_{25}$-modules in $\ind(\cO_1^0\otimes \cO_{25}^0)$. Then $\rep^0 A$ is the braided tensor category of local $A$-modules in $\ind(\cO_1^0\otimes \cO_{25}^0)$. Let $\cO_A$ denote the category of finite-length $A$-modules that restrict to $V_1\otimes V_{25}$-modules in $\ind(\cO_1^0\otimes \cO_{25}^0)$. Obviously, $\cO_A$ is a subcategory of $\rep^0 A$. 

Similar to the previous section, there is an induction functor
$\cF_A: \cO_1^0\otimes \cO_{25}^0 \rightarrow  \rep A$ satisfying Frobenius reciprocity:
\begin{equation}\label{Frob}
\hom_{A}(\cF_A(M), W) \xrightarrow{\cong} \hom_{V_1\otimes V_{25}}(M, W)
\end{equation}
for any object $M$ of $\cO_1^0 \otimes \cO_{25}^0$ and $W$ of $\rep A$. To simplify notation, we denote the simple objects of $\cO_1^0\otimes\cO_{25}^0$ by
\[
M_{r,r'} : = \cL^{(1)}_{r,1} \otimes \cL_{r',1}^{(25)}.
\]
for $r,r' \in \ZZ_+$.
\begin{thm}\label{thm:chiral_univ_center_rep_theory}
Properties of the categories $\repA$ and $\cO_A$ are as follows:
\begin{itemize}
\item[(1)] The category $\repA$ is semisimple, and the induced modules $\cW_r : = \cF_A(M_{r,1})$, $r \in \ZZ_+$, are simple and exhaust all the simple objects in $\rep A$ up to isomorphism.
\item[(2)] The tensor products of simple modules are as follows: for $r, r' \in \ZZ_+$, 
\begin{equation}
\cW_r \tens \cW_{r'} \cong \bigoplus_{\substack{k=\vert r-r'\vert+1\\ k+r+r'\equiv 1\,(\mathrm{mod}\,2)}}^{r+r'-1} \cW_{k}.
\end{equation} 
\item[(3)] Induced modules have the following decompositions: for $r,r' \in \ZZ_+$, 
\begin{equation}\label{eqn:induced}
\cF_A(M_{r,r'}) \cong \bigoplus_{\substack{k=\vert r-r'\vert+1\\ k+r+r'\equiv 1\,(\mathrm{mod}\,2)}}^{r+r'-1} \cW_{k}.
\end{equation}

\item[(4)] The category $\cO_A$ has simple objects $\cW_{2n+1}$, $n \in \NN$, and is braided  tensor equivalent to $\rep PSL_2(\CC)$. In particular, $\cO_A$ is rigid and symmetric.
\end{itemize}
\end{thm}
\begin{proof}
Since $\cF_A$ is monoidal, the $\mathfrak{sl}_2$-type fusion rules of $V_1$-modules imply
\begin{align*}
\cW_r \tens \cW_{r'} = \cF_{A}(M_{r,1}) \tens \cF_A(M_{r',1}) \cong \cF_A\bigg(\bigoplus_{\substack{k=\vert r-r'\vert+1\\ k+r+r'\equiv 1\,(\mathrm{mod}\,2)}}^{r+r'-1} M_{k,1}\bigg) = \bigoplus_{\substack{k=\vert r-r'\vert+1\\ k+r+r'\equiv 1\,(\mathrm{mod}\,2)}}^{r+r'-1} \cW_{k},
\end{align*}
proving (2). To prove (3), first note $\cF_A(M_{r,r'}) \cong \cF_A(M_{r,1}) \tens \cF_A(M_{1, r'})$ because $\cF_A$ is monoidal. Then because $\cO_1^0$ and $\cO_{25}^0$ are semisimple ribbon tensor categories and the modules $\cL_{r,1}^{(i)}$ for $i=1,25$ are self-dual, \cite[Key Lemma~4.2]{CKM2} (see also \cite[Remark 4.3]{CKM2}) shows that $\cF_A(M_{r,1}) \cong \cF_A(M_{1,r})$. Consequently,
\begin{equation*}
\cF_A(M_{r,r'}) \cong \cF_A(M_{r,1}) \tens \cF_A(M_{1, r'}) \cong \cW_r \tens \cW_{r'} \cong \bigoplus_{\substack{k=\vert r-r'\vert+1\\ k+r+r'\equiv 1\,(\mathrm{mod}\,2)}}^{r+r'-1} \cW_{k}.
\end{equation*}

To prove (1), Frobenius reciprocity \eqref{Frob} and the $\mathfrak{sl}_2$-type fusion rules of $V_1$-modules imply
\begin{align}
\hom_A(\cW_r, \cW_s) &= \hom_A(\cF_A(M_{r,1}), \cF_A(M_{s,1}))\nonumber\\ 
&\cong\hom_{V_1\otimes V_{25}}\bigg(M_{r,1},\bigoplus_{r' \in \ZZ_+}\bigoplus_{\substack{k=\vert s-r'\vert+1\\ k+s+r'\equiv 1\,(\mathrm{mod}\,2)}}^{s+r'-1} M_{k,r'}\bigg)
\cong \delta_{r,s}\cdot\CC. \label{schur} 
\end{align}
Now if $W\subseteq\cW_r$ is a non-zero submodule, then $W$ contains an irreducible $V_1\otimes V_{25}$-submodule, say $M_{r', r''}$ for $r', r'' \in \ZZ_+$. By \eqref{Frob}, there is a non-zero $A$-module map $\cF_A(M_{r', r''}) \rightarrow W \hookrightarrow \cW_r$. So because $\cF_A(M_{r', r''})$ is a direct sum of certain $\cW_k$, $k \in \ZZ_+$, by \eqref{eqn:induced}, and because of \eqref{schur}, $\cF_A(M_{r', r''})$ must contain a copy of $\cW_r$ and furthermore there is a non-zero composition $\cW_r \rightarrow W \hookrightarrow \cW_r$. By \eqref{schur}, this $A$-module map is a multiple of the identity and is in particular surjective. Thus $W=\cW_r$, proving $\cW_r$ is a simple $A$-module.

Now take an arbitrary object $X$ of $\rep A$; it is a direct sum of irreducible $V_1\otimes V_{25}$-modules $M_{r,r'}$, $r,r'\in\ZZ_+$, since it is an object of $\ind(\cO_1^0\otimes\cO_{25}^0)$. By Frobenius reciprocity, $X$ is a sum of quotients of induced modules $\cF_A(M_{r,r'})$, and it then follows from the decomposition \eqref{eqn:induced} that $X$ is the sum, and thus also the direct sum, of submodules $\cW_r$ for various $r \in \ZZ_{+}$. This completes the proof of (1).

To prove (4), \cite[Lemma~2.65]{CKM1} shows that a simple module $\cW_r$ is an object of $\cO_A$ if and only if the monodromy isomorphism $\cM_{A, M_{r,1}}$ is the identity on $A\tens M_{r,1}$. For $r' \in \ZZ_+$,
\[
M_{r,1} \tens M_{r',r'} = \bigoplus_{\substack{k=\vert r-r'\vert+1\\ k+r+r'\equiv 1\,(\mathrm{mod}\,2)}}^{r+r'-1} M_{k,r'}.
\]
Thus by the balancing equation and conformal weights, $\cM_{A, M_{r,1}} = \Id_{A\tens M_{r,1}}$ if and only if
\[
h_{r,1}^{(1)} + h^{(1)}_{r',1} - h^{(1)}_{k,1} \in \ZZ
\]
for all $r' \in \ZZ_+$ and $|r-r'|+1 \leq k \leq r+r'-1$ such that $k+r+r'\equiv 1\,(\mathrm{mod}\,2)$. It turns out that $r \equiv 1\, (\mathrm{mod}\, 2)$.

Finally, by Proposition \ref{prop:braiding2}, the subcategory $\langle\cL_{2n+1,1}^{(1)}\,\vert\,n\in\NN\rangle\subseteq\cO_1^0$ containing the indicated simple objects is braided tensor equivalent to $\rep PSL_2(\CC)$.
Also, the composition
\begin{equation*}
\langle\cL_{2n+1,1}^{(1)}\,\vert\,n\in\NN\rangle \xrightarrow{M\mapsto M\otimes\cL_{1,1}^{(25)}}\cO_1^0\otimes\cO_{25}^0\xrightarrow{\cF_A} \repA
\end{equation*}
is a braided tensor functor (see for example \cite[Theorem 2.67]{CKM1}) which is fully faithful \cite[Lemma 6.1]{CKM2} and essentially surjective onto $\cO_A$. Thus $\cO_A \cong \rep PSL_2(\CC)$ as braided tensor categories.
\end{proof}

\begin{rem}
As in the last paragraph of the above proof, induction and Proposition \ref{prop:braiding2} show that the subcategory of finite-length objects in $\repA$ is tensor equivalent to $\cC(-1,\mathfrak{sl}_2)$. Since the entire tensor category $\repA$ is not naturally braided (although it is braidable since $\cC(-1,\mathfrak{sl}_2)$ is), we do not view induction as a braided tensor functor in this case.
\end{rem}

\begin{rem}\label{rem:generic_ch_univ_cent}
We can similarly determine the representation theory of $\mathbf{I}_{SL_2}^k$ at generic level, that is, $k=-2+t$ where $t\notin\mathbb{Q}$. In this case, $\mathbf{I}_{SL_2}^k$ is an extension of $V_{c(t)}\otimes V_{c(-t)}$, where $c(\pm t)=13\mp 6t\mp 6t^{-1}$, and
\begin{equation*}
\mathbf{I}_{SL_2}^k\cong\bigoplus_{s\in\ZZ_+} \cL_{1,s}^{(c(t))}\otimes\cL_{1,s}^{(c(-t))}
\end{equation*}
as a $V_{c(t)}\otimes V_{c(-t)}$-module. Moreover, $\cO_{c(\pm t)}$ are rigid semsimple braided tensor categories \cite{CJORY} with $\mathfrak{sl}_2$-type fusion rules \cite{FZ2}, and the Deligne product category $\cO_{c(t)}\otimes\cO_{c(-t)}$ has simple objects $M_{(r,s),(r',s')}:=\cL_{r,s}^{c(t)}\otimes\cL_{r',s'}^{c(-t)}$. Then the following results have proofs similar to Theorem \ref{thm:chiral_univ_center_rep_theory}:
\begin{enumerate}
\item The category $\rep \mathbf{I}_{SL_2}^k$ of (possibly non-local) $\mathbf{I}_{SL_2}^k$-modules which restrict to $V_{c(t)}\otimes V_{c(-t)}$-modules in $\ind(\cO_{c(t)}\otimes\cO_{c(-t)})$ is semisimple, and the induced modules $\cW_{r,r'}^{s}:=\mathcal{F}_{\mathbf{I}_{SL_2}^k}(M_{(r,s),(r',1)})$ for $r,r',s\in\ZZ_+$ are simple and exhaust the simple objects of $\rep\mathbf{I}_{SL_2}^k$ up to isomorphism.

\item For $r_1,r_1',s_1,r_2,r_2',s_2\in\ZZ_+$,
\begin{equation*}
\cW_{r_1,r_1'}^{s_1}\tens\cW_{r_2,r_2'}^{s_2}\cong\bigoplus_{\substack{k=\vert r_1-r_2\vert+1\\ k+r_1+r_2\equiv 1\,(\mathrm{mod}\,2)}}^{r_1+r_2-1}\bigoplus_{\substack{k'=\vert r_1'-r_2'\vert+1\\ k'+r_1'+r_2'\equiv 1\,(\mathrm{mod}\,2)}}^{r_1'+r_2'-1}\bigoplus_{\substack{\ell=\vert s_1-s_2\vert+1\\ \ell+s_1+s_2\equiv 1\,(\mathrm{mod}\,2)}}^{s_1+s_2-1} \cW_{k,k'}^{\ell}.
\end{equation*}

\item For $r,s,r',s'\in\ZZ_+$,
\begin{equation*}
\cF_{\mathbf{I}^k_{SL_2}}(M_{(r,s),(r',s')})\cong\bigoplus_{\substack{\ell=\vert s-s'\vert+1\\ \ell+s+s'\equiv 1\,(\mathrm{mod}\,2)}}^{s+s'-1} \cW_{r,r'}^\ell.
\end{equation*}

\item The category of finite-length (local) $\mathbf{I}^k_{SL_2}$-modules which restrict to $V_{c(t)}\otimes V_{c(-t)}$-modules in $\ind(\cO_{c(t)}\otimes\cO_{c(-t)})$ is a rigid semisimple braided tensor category with simple objects $\cW_{r,r'}^1$ such that $r\equiv r'\,(\mathrm{mod}\,2)$.

\end{enumerate}
Note that the subcategory of finite-length objects in $\rep\mathbf{I}_{SL_2}^k$ for $k=-2+t$, $t\notin\QQ$, is the Deligne product of three tensor subcategories with $\mathfrak{sl}_2$-type fusion rules. Via the induction functor, these three subcategories are tensor equivalent to 
\begin{equation*}
\cO_{c(t)}^L:=\langle \cL_{r,1}^{(c(t))}\,\vert\,r\in\ZZ_+\rangle,\quad\cO_{c(-t)}^L:=\langle\cL_{r',1}^{(c(-t))}\,\vert\,r'\in\ZZ_+\rangle,\quad\cO_{c(t)}^R:=\langle\cL_{1,s}^{(c(t))}\,\vert\,s\in\ZZ_+\rangle.
\end{equation*}
All three of these categories are equivalent to quantum group categories $\cC(q,\mathfrak{sl}_2)$ by \cite[Theorem $A_{\infty}$]{KW}. We could determine the respective values of $q$ by calculating the intrinsic dimension of the generating simple object in each category. Alternatively, we can use the heterogeneous vertex operator algebras of \cite[Theorem 4.3]{FS} (called \textit{equivariant affine $W$-algebras} in \cite[Section 6]{Ar}) and the (braid-reversing) tensor equivalence of \cite[Main Theorem 2]{CKM2} to show that these three Virasoro categories are tensor equivalent to Kazhdan-Lusztig categories for affine $\mathfrak{sl}_2$ at suitable levels, which in turn are equivalent to $\cC(q,\mathfrak{sl}_2)$ at suitable $q$ \cite{KL}. For example, $\cO_{c(t)}^L$ will be tensor equivalent to the Kazhdan-Lusztig category for affine $\mathfrak{sl}_2$ at level $-2-t^{-1}$, which is equivalent to $\cC(e^{-\pi i t},\mathfrak{sl}_2)\cong\cC(e^{\pi i t},\mathfrak{sl}_2)$. These tensor equivalences are also given in \cite[Proposition 5.5.2]{CJORY}, except that we need to make two corrections in \cite[Section~5.3]{CJORY}: $t = k+2$ should be changed to $t^{-1} = k+2$ and the $W$-algebra module ${\rm \bf L}_k(\chi_{\mu - 2(k+h^{\vee})\nu})$ should be ${\rm \bf L}_k(\chi_{\mu - (k+h^{\vee})\nu})$.
\end{rem}

\appendix

\section{Algebras in equivalent tensor categories}\label{app:alg_in_tens_cat}

The tensor-categorical results in this appendix are straightforward and certainly known, but we include some details from their proofs to make this paper more self-contained. We also use this appendix to recall notation from, for example, \cite{KO, CKM1} for algebras in tensor categories and their modules.

Let $\cC$ be a braided tensor category with tensor product bifunctor $\tens$ and unit object $\vac$. We recall (from \cite{KO, CKM1}, for example) that a (commutative) $\cC$-algebra $(A,\mu_A,\iota_A)$ is an object $A$ of $\cC$ equipped with multiplication and unit morphisms
\begin{equation*}
\mu_A: A\tens A\longrightarrow A,\qquad\iota_A: \vac\longrightarrow A
\end{equation*}
which satisfy natural unit and associativity (and commutativity) axioms. An isomorphism between two $\cC$-algebras $(A,\mu_A,\iota_A)$ and $(B,\mu_B,\iota_B)$ is a $\cC$-morphism $f: A\rightarrow B$ such that $f\circ\mu_A =\mu_B\circ(f\tens f)$ and $f\circ\iota_A=\iota_B$.

 Given a commutative algebra $A$, the tensor category $\repA$ of (left) $A$-modules has objects $(X,\mu_X)$ where $X$ is an object of $\cC$ and $\mu_X: A\tens X\rightarrow X$ is a morphism satisfying natural left unit and associativity axioms. The category $\repA$ has a braided tensor subcategory $\rep^0A$ consisting of all objects $(X,\mu_X)$ such that
\begin{equation*}
\mu_X\circ\cM_{A,X}=\mu_X,
\end{equation*}
where $\cM_{A,X}:=\cR_{X,A}\circ\cR_{A,X}$ is the natural double braiding, or monodromy, isomorphism in $\cC$. Given an object $(X,\mu_X)$ of $\repA$, an $A$-submodule of $X$ is an object $(W,\mu_W)$ of $\repA$ equipped with a $\repA$-injection $i: W\rightarrow X$. We say that $(X,\mu_X)$ is simple if every submodule $i: W\rightarrow X$ is either $0$ or an isomorphism. If $A$ is commutative, we can identify ideals of $A$ with $A$-submodules of $A$, so that $A$ is simple as a $\cC$-algebra if and only if it simple as an $A$-module.

Let $\cF:\cC\rightarrow\cD$ be a tensor functor, so that there is an isomorphism $\varphi:\vac_\cD\rightarrow\cF(\vac_\cC)$ and a natural isomorphism
\begin{equation}\label{eqn:F}
F: \tens_\cD\circ(\cF\times\cF)\longrightarrow\cF\circ\tens_\cC
\end{equation}
which are suitably compatible with the unit and associativity isomorphisms of $\cC$ and $\cD$. If $(A,\mu_A,\iota_A)$ is a $\cC$-algebra, then $(\cF(A),\mu_{\cF(A)},\iota_{\cF(A)})$ is a $\cD$-algebra, where
\begin{equation*}
\mu_{\cF(A)}: \cF(A)\tens_\cD\cF(A)\xrightarrow{F_{A,A}} \cF(A\tens_\cC A) \xrightarrow{\cF(\mu_A)} \cF(A)
\end{equation*}
and
\begin{equation*}
\iota_{\cF(A)}: \vac_\cD\xrightarrow{\varphi}\cF(\vac_\cC)\xrightarrow{\cF(\iota_A)} \cF(A).
\end{equation*}
If $A$ is commutative and $\cF$ is braided or braid-reversing, then $\cF(A)$ is also commutative. 

If $(X,\mu_X)$ is an object of $\repA$, then $(\cF(X),\mu_{\cF(X)})$ is an object of $\rep \cF(A)$, where
\begin{equation*}
\mu_{\cF(X)}: \cF(A)\tens_\cD\cF(X)\xrightarrow{F_{A,X}} \cF(A\tens_\cC X)\xrightarrow{\cF(\mu_X)} \cF(X).
\end{equation*}
Moreover, if $f: X_1\rightarrow X_2$ is a morphism in $\repA$, then $\cF(f): \cF(X_1)\rightarrow\cF(X_2)$ is a morphism in $\rep\cF(A)$. If $A$ is commutative, $\cF$ is braided or braid-reversing, and $(X,\mu_X)$ is an object of $\rep^0A$, then $(\cF(X),\mu_{\cF(X)})$ is an object of $\rep^0\cF(A)$. Thus if $\cF$ is braided or braid-reversing, it restricts to a functor from $\rep^0 A$ to $\rep^0\cF(A)$. In the case that $\cF$ is an equivalence of categories, we get:
\begin{prop}\label{prop:F_equiv_on_Rep_A}
Let $\cF:\cC\rightarrow\cD$ be a braided or braid-reversing tensor equivalence. If $A$ is a commutative $\cC$-algebra, then $\cF:\rep^0 A\rightarrow\rep^0\cF(A)$ is an equivalence of categories.
\end{prop}
\begin{proof}
To show that $\cF$ is essentially surjective onto $\rep^0\cF(A)$, suppose $(W,\mu_W)$ is an object of $\rep^0\cF(A)$. Since $\cF$ is essentially surjective onto $\cD$,  there is an isomorphism $f: W\rightarrow\cF(X)$ in $\cD$ for some object $X$ in $\cC$, and then $f: (W,\mu_W)\rightarrow(\cF(X),\mu_{\cF(X)})$ is an isomorphism in $\rep^0\cF(A)$ where 
\begin{equation*}
\mu_{\cF(X)}=f\circ\mu_W\circ(\Id_{\cF(A)}\tens_\cD f^{-1}).
\end{equation*}
 Using the full faithfulness of $\cF$, we define $\mu_X: A\tens_\cC X\rightarrow X$ to be the unique $\cC$-morphism such that 
\begin{equation*}
\cF(\mu_X)=\mu_{\cF(X)}\circ F_{A,X}^{-1}.
\end{equation*}
It is then straightforward to check that $(X,\mu_X)$ is an object of $\repA$ such that $\cF(X,\mu_X)=(\cF(X),\mu_{\cF(X)})\cong(W,\mu_W)$. For example, to prove the associativity of $\mu_X$, the faithfulness of $\cF$ implies it is enough to show
\begin{equation*}
\cF(\mu_X\circ(\Id_A\tens_\cC\mu_X)) = \cF(\mu_X\circ(\mu_A\tens_\cC\Id_X)\circ\cA_{A,A,X}),
\end{equation*}
which follows from the definitions and the associativity of $(\cF(X),\mu_{\cF(X)})$. Similarly, since $W$, equivalently $\cF(X)$, is an object of $\rep^0\cF(A)$, then
\begin{equation*}
\cF(\mu_X)=\cF(\mu_X\circ\cM_{A,X}^{\pm 1})
\end{equation*}
since $\cF$ is braided or braid-reversing, and it follows that $X$ is an object of $\rep^0A$.

Now since $\cF: \cC\rightarrow\cD$ is faithful, $\cF$ remains faithful when restricted to $\rep^0A$. Then to show that the restriction to $\rep^0 A$ is full, let $g: \cF(X_1)\rightarrow\cF(X_2)$ be a morphism in $\rep^0\cF(A)$, where $X_1$ and $X_2$ are objects of $\rep^0 A$. Then $g=\cF(f)$ for some morphism $f: X_1\rightarrow X_2$ in $\cC$, and 
\begin{align*}
\cF(f\circ\mu_{X_1}) & =g\circ\mu_{\cF(X_1)}\circ F_{A,X_1}^{-1} =\mu_{\cF(X_2)}\circ(\Id_{\cF(A)}\tens_\cD g)\circ F_{A,X_1}^{-1}\nonumber\\
& = \mu_{\cF(X_2)}\circ F_{A,X_2}^{-1}\circ\cF(\Id_A\tens_\cC f) =\cF(\mu_{X_2}\circ(\Id_A\tens_\cC f)).
\end{align*}
Since $\cF$ is faithful, it follows that $f$ is a morphism in $\rep^0 A$.
\end{proof}

Since equivalences of abelian categories preserve zero objects and thus also zero morphisms, if $\cF:\cC\rightarrow\cD$ is an equivalence, then a morphism $i$ in $\cC$ is injective if and only if $\cF(i)$ is injective in $\cD$. Consequently, $\cF$ also preserves simple objects. Thus:
\begin{cor}\label{cor:F_preserve_simple}
Let $\cF:\cC\rightarrow \cD$ be a braided or braid-reversing tensor functor. If $A$ is a commutative $\cC$-algebra, then an object $X$ of $\rep^0 A$ is simple if and only if $\cF(X)$ is a simple object of $\rep^0\cF(A)$. In particular, $A$ is a simple $\cC$-algebra if and only if $\cF(A)$ is a simple $\cD$-algebra.
\end{cor}

The next result is that if $\cF: \cC\rightarrow\cD$ is a braided or braid-reversing tensor functor and $A$ is a commutative $\cC$-algebra, then $\cF$ is also braided or braid-reversing on $\rep^0A$. The proof is straightforward but somewhat long, so we only include the more non-trivial details:
\begin{thm}\label{thm:F_braid_tens_on_Rep_A}
Let $\cF: \cC\rightarrow\cD$ be a right exact braided, respectively braid-reversing, tensor functor. If $A$ is a commutative $\cC$-algebra, then $\cF:\rep^0 A\rightarrow\rep^0\cF(A)$ has the structure of a braided, respectively braid-reversing, tensor functor.
\end{thm}
\begin{proof}
Letting $\tens_A$ and $\tens_{\cF(A)}$ denote the tensor products in $\rep^0 A$ and $\rep^0\cF(A)$, respectively, we need to obtain a natural isomorphism
\begin{equation*}
F^A: \tens_{\cF(A)}\circ(\cF\times\cF)\longrightarrow\cF\circ\tens_A.
\end{equation*}
from the original natural isomorphism $F$ of \eqref{eqn:F}. To do so, we recall (from \cite[Section 2.3]{CKM1} for example) that for objects $X_1$ and $X_2$ in $\rep^0 A$, $X_1\tens_A X_2$ is the cokernel of $\mu^{(1)}-\mu^{(2)}$, where $\mu^{(1)}=\mu_{X_1}\tens_\cC\Id_{X_2}$ and $\mu^{(2)}$ is the composition
\begin{align*}
(A\tens_\cC X_1)\tens_{\cC} X_2 & \xrightarrow{\cR_{A,X_1}\tens_\cC\Id_{X_2}} (X_1\tens_\cC A)\tens_{\cC} X_2\nonumber\\
& \xrightarrow{\cA_{X_1,A,X_2}^{-1}} X_1\tens_\cC (A\tens_\cC X_2)\xrightarrow{\Id_{X_1}\tens_\cC\mu_{X_2}} X_1\tens_\cC X_2.
\end{align*}
We use $\eta_{X_1,X_2}: X_1\tens_\cC X_2\rightarrow X_1\tens_A X_2$ to denote the cokernel morphism. Similar definitions and notation apply to objects in $\rep^0\cF(A)$. We would like to define morphisms $F^A_{X_1,X_2}$ and $G^A_{X_1,X_2}$ in $\cD$ such that the diagrams
\begin{equation*}
\xymatrixcolsep{4pc}
\xymatrix{
\cF(X_1)\tens_\cD\cF(X_2) \ar[d]^{\eta_{\cF(X_1),\cF(X_2)}} \ar[r]^(.525){F_{X_1,X_2}} & \cF(X_1\tens_\cC X_2) \ar[d]^{\cF(\eta_{X_1,X_2})} \\
\cF(X_1)\tens_{\cF(A)} \cF(X_2) \ar[r]^(.55){F^A_{X_1,X_2}} & \cF(X_1\tens_A X_2) \\
}
\end{equation*}
and
\begin{equation*}
\xymatrixcolsep{4pc}
\xymatrix{
 \cF(X_1\tens_\cC X_2) \ar[d]^{\cF(\eta_{X_1,X_2})} \ar[r]^(.475){F^{-1}_{X_1,X_2}} &  \cF(X_1)\tens_\cD\cF(X_2)  \ar[d]^{\eta_{\cF(X_1),\cF(X_2)}} \\
 \cF(X_1\tens_A X_2) \ar[r]^(.45){G^A_{X_1,X_2}} & \cF(X_1)\tens_{\cF(A)} \cF(X_2)  \\
}
\end{equation*}
commute.

\allowdisplaybreaks

From the cokernel definitions of $\cF(X_1)\tens_A\cF(X_2)$ and $X_1\tens_A X_2$ and by the right exactness of $\cF$ (so that $\cF(X_1\tens_A X_2)$ is still a cokernel), the existence and uniqueness of $F^A_{X_1,X_2}$ and $G^A_{X_1,X_2}$ is equivalent to the identities
\begin{align*}
\cF(  \eta_{X_1,X_2})\circ F_{X_1,X_2}\circ\mu^{(1)} & =\cF(\eta_{X_1,X_2})\circ F_{X_1,X_2}\circ\mu^{(2)},\\
\eta_{\cF(X_1),\cF(X_2)}\circ F^{-1}_{X_1,X_2}\circ\cF(\mu^{(1)}) & =\eta_{\cF(X_1),\cF(X_2)}\circ F^{-1}_{X_1,X_2}\circ\cF(\mu^{(2)}).
\end{align*}
When $\cF$ is braid-reversing (the more interesting case), the first identity is proved as follows:
\begin{align*}
\cF(& \eta_{X_1,X_2}) \circ F_{X_1,X_2}\circ(\Id_{\cF(X_1)}\tens_\cD\mu_{\cF(X_2)})\circ\cA_{\cF(X_1),\cF(A),\cF(X_2)}^{-1}\circ(\cR_{\cF(A),\cF(X_1)}\tens_\cD\Id_{\cF(X_2)})\nonumber\\
& = \cF(\eta_{X_1,X_2})  \circ F_{X_1,X_2}\circ(\Id_{\cF(X_1)}\tens_\cD\cF(\mu_{X_2}))\circ(\Id_{\cF(X_1)}\tens_\cD F_{A,X_2})\circ\nonumber\\
&\hspace{6em}\circ\cA_{\cF(X_1),\cF(A),\cF(X_2)}^{-1}\circ(\cR_{\cF(A),\cF(X_1)}\tens_\cD\Id_{\cF(X_2)})\nonumber\\
& =\cF(\eta_{X_1,X_2}\circ(\Id_{X_1}\tens_\cC\mu_{X_2}))\circ F_{X_1,A\tens_\cC X_2}\circ(\Id_{\cF(X_1)}\tens_\cD F_{A,X_2})\circ\nonumber\\
&\hspace{6em}\circ\cA_{\cF(X_1),\cF(A),\cF(X_2)}^{-1}\circ(\cR_{\cF(A),\cF(X_1)}\tens_\cD\Id_{\cF(X_2)})\nonumber\\
& =\cF(\eta_{X_1,X_2}\circ(\Id_{X_1}\tens_\cC\mu_{X_2})\circ\cA_{X_1,A,X_2}^{-1}\circ(\cR^{-1}_{X_1,A}\tens_\cC\Id_{X_2}))\circ\nonumber\\
&\hspace{6em}\circ F_{A\tens_\cC X_1,X_2}\circ(F_{A,X_1}\tens_\cD\Id_{\cF(X_2)})\nonumber\\
& =\cF(\eta_{X_1,X_2}\circ(\Id_{X_1}\tens_\cC\mu_{X_2})\circ\cA_{X_1,A,X_2}^{-1}\circ(\cR_{A,X_1}\tens_\cC\Id_{X_2}))\circ\nonumber\\
&\hspace{6em}\circ\cF(\cM_{A,X_1}^{-1}\tens_\cC\Id_{X_2})\circ F_{A\tens_\cC X_1,X_2}\circ(F_{A,X_1}\tens_\cD\Id_{\cF(X_2)})\nonumber\\
& =\cF(\eta_{X_1,X_2}\circ(\mu_{X_1}\tens_\cC\Id_{X_2})\circ(\cM_{A,X_1}^{-1}\tens\Id_{X_2}))\circ F_{A\tens_\cC X_1,X_2}\circ(F_{A,X_1}\tens_\cD\Id_{\cF(X_2)})\nonumber\\
& =\cF(\eta_{X_1,X_2})\circ F_{X_1,X_2}\circ(\cF(\mu_{X_1})\tens_\cC \Id_{X_2})\circ(F_{A,X_1}\tens_\cD\Id_{\cF(X_2)})\nonumber\\
& =\cF(\eta_{X_1,X_2})\circ F_{X_1,X_2}\circ\mu^{(1)},
\end{align*}
where the next to last equality uses the assumption that $X_1$ is an object of $\rep^0A$. The second identity is proved similarly, so we get the desired morphism $F^A_{X_1,X_2}$ in $\cD$ with inverse $G^A_{X_1,X_2}$.

To prove that $F^A_{X_1,X_2}$ is an isomorphism in $\rep^0\cF(A)$, we use the definitions of $\mu_{X_1\tens_A X_2}$ and $\mu_{\cF(X_1)\tens_{\cF(A)}\cF(X_2)}$ from \cite[Section 2.3]{CKM1}, along with compatibility of $F$ with the associativity isomorphisms, to calculate
\begin{align*}
&\mu_{\cF(X_1\tens_A X_2)}  \circ(\Id_{\cF(A)}\tens_\cD F^A_{X_1,X_2})\circ(\Id_{\cF(A)}\tens_\cD\eta_{\cF(X_1),\cF(X_2)})\nonumber\\
& \quad=\cF(\mu_{X_1\tens_A X_2})\circ F_{A,X_1\tens_A X_2}\circ(\Id_{\cF(A)}\tens_\cD\cF(\eta_{X_1,X_2}))\circ(\Id_{\cF(A)}\tens_\cD F_{X_1,X_2})\nonumber\\
& \quad	=\cF(\mu_{X_1\tens_A X_2}\circ(\Id_A\tens_\cC\eta_{X_1,X_2}))\circ F_{A,X_1\tens_\cC X_2}\circ(\Id_{\cF(A)}\tens_\cD F_{X_1,X_2})\nonumber\\
& \quad=\cF(\eta_{X_1,X_2})\circ\cF(\mu_{X_1}\tens_\cC\Id_{X_2})\circ\cF(\cA_{A,X_1,X_2})\circ F_{A,X_1\tens_\cC X_2}\circ(\Id_{\cF(A)}\tens_\cD F_{X_1,X_2})\nonumber\\
& \quad=\cF(\eta_{X_1,X_2})\circ\cF(\mu_{X_1}\tens_\cC\Id_{X_2})\circ F_{A\tens_\cC X_1,X_2}\circ(F_{A,X_1}\tens_\cD\Id_{\cF(X_2)})\circ\cA_{\cF(A),\cF(X_1),\cF(X_2)}\nonumber\\
& \quad=\cF(\eta_{X_1,X_2})\circ F_{X_1,X_2}\circ(\cF(\mu_{X_1})\tens_\cD\Id_{\cF(X_2)})\circ(F_{A,X_1}\tens_\cD\Id_{\cF(X_2)})\circ\cA_{\cF(A),\cF(X_1),\cF(X_2)}\nonumber\\
& \quad= F^A_{X_1,X_2}\circ\eta_{\cF(X_1),\cF(X_2)}\circ(\mu_{\cF(X_1)}\tens_\cD\Id_{\cF(X_2)})\circ\cA_{\cF(A),\cF(X_1),\cF(X_2)}\nonumber\\
& \quad= F^A_{X_1,X_2}\circ\mu_{\cF(X_1)\tens_{\cF(A)}\cF(X_2)}\circ(\Id_{\cF(A)}\tens_\cD\eta_{\cF(X_1),\cF(X_2)}).
\end{align*}
Because $\Id_{\cF(A)}\tens_\cD\eta_{\cF(X_1),\cF(X_2)}$ is surjective, it follows that 
\begin{equation*}
\mu_{\cF(X_1\tens_A X_2)}  \circ(\Id_{\cF(A)}\tens_\cD F^A_{X_1,X_2}) =F^A_{X_1,X_2}\circ\mu_{\cF(X_1)\tens_{\cF(A)}\cF(X_2)}
\end{equation*}
as required. The proof that $F^A$ defines a natural transformation is also straightforward from the definitions in \cite[Section 2.3]{CKM1} and the fact that $F$ is a natural transformation.

To show that the natural isomorphism $F^A$ gives $\cF:\rep^0 A\rightarrow\rep^0\cF(A)$ the structure of a braided (or braid-reversing) tensor functor, we also need $F^A$ to be compatible with the unit, associativity and braiding isomorphisms. Since the compatibility proofs are straightforward, we only discuss compatibility with the right unit and braiding isomorphisms in the case that $\cF$ is braid-reversing. For the right unit isomorphisms, we need to show
\begin{equation*}
\cF(r^A_X)\circ F^A_{X,A}=r_{\cF(X)}^{\cF(A)}
\end{equation*}
for any object $X$ of $\rep^0 A$, where $r^A$ and $r^{\cF(A)}$ are the right unit isomorphisms in $\rep^0 A$ and $\rep^0\cF(A)$, respectively. Indeed, from the definitions in \cite[Section 2.3]{CKM1} and the assumption that $X$ is an object of $\rep^0A$, we get
\begin{align*}
\cF(r^A_X)\circ F^A_{X,A}\circ\eta_{\cF(X),\cF(A)} & = \cF(r^A_X)\circ\cF(\eta_{X,A})\circ F_{X,A}\nonumber\\
& =\cF(\mu_X)\circ\cF(\cR_{A,X}^{-1})\circ F_{X,A}\nonumber\\
& =\cF(\mu_X)\circ\cF(\cR_{X,A})\circ F_{X,A}\nonumber\\
& =\cF(\mu_X)\circ F_{A,X}\circ\cR_{\cF(A),\cF(X)}^{-1}\nonumber\\
& =\mu_{\cF(X)}\circ\cR_{\cF(A),\cF(X)}^{-1} = r^{\cF(A)}_{\cF(X)}\circ\eta_{\cF(X),\cF(A)}.
\end{align*}
So the desired equality follows from surjectivity of $\eta_{\cF(X),\cF(A)}$. Similarly, the braid-reversing properties of $F$ and the definitions in \cite[Section 2.6]{CKM2} show that 
\begin{equation*}
\cF(\cR^A_{X_1,X_2})\circ F^A_{X_1,X_2}= F^A_{X_2,X_1}\circ\cR^{-1}_{\cF(X_2),\cF(X_1)}
\end{equation*}
for objects $X_1$ and $X_2$ of $\rep^0 A$, so that $\cF$ defines a braid-reversed tensor functor from $\rep^0 A$ to $\rep^0\cF(A)$.
\end{proof}

Since equivalences between abelian categories are automatically exact, we get the following corollary of Proposition \ref{prop:F_equiv_on_Rep_A} and Theorem \ref{thm:F_braid_tens_on_Rep_A}:
\begin{cor}\label{cor:rep_0_A_rep_0_F(A)}
Let $\cF:\cC\rightarrow\cD$ be a braided, respectively braid-reversing, tensor equivalence. If $A$ is a commutative $\cC$-algebra, then $\cF$ induces a braided, respectively braid-reversing, tensor equivalence between $\rep^0 A$ and $\rep^0\cF(A)$.
\end{cor}

\section{Uniqueness of \texorpdfstring{$L_1(\mathfrak{sl}_2)$}{L1(sl2)}}\label{app:L1_sl_2_unique}

In this appendix, we present some uniqueness results for the simple affine vertex operator algebra $L_1(\mathfrak{sl}_2)$ associated to $\mathfrak{sl}_2$ at level $1$, and also for its non-trivial irreducible module. These results are perhaps known to experts, but we provide proofs here for completeness. The automorphism group of $L_1(\mathfrak{sl}_2)$ is $PSL_2(\CC)$, with fixed-point subalgebra the Virasoro vertex operator algebra $V_1$ \cite{DG, Mi}. As a $PSL_2(\CC)\times V_1$-module, 
\begin{equation}\label{eqn:L1_sl2_decomp}
L_1(\mathfrak{sl}_2)\cong\bigoplus_{n=0}^\infty V(2n)\otimes\cL_{2n+1,1}.
\end{equation} 
We first establish the uniqueness of a simple vertex operator algebra with this decomposition as a $V_1$-module, while ignoring $PSL_2(\CC)$-actions:

\begin{thm}\label{thm:L1_sl_2_unique}
If $V$ is a simple vertex operator algebra extension of $V_1$ such that $V\cong\bigoplus_{n=0}^\infty V(2n)\otimes\cL_{2n+1,1}$ as $V_1$-modules (with $V(2n)$ considered just as a $(2n+1)$-dimensional vector space), then $V\cong L_1(\mathfrak{sl}_2)$ as vertex operator algebras.
\end{thm}

\begin{proof}
For $n\in\NN$, we use $v_{2n+1,1}$ to denote a generating vector of $\cL_{2n+1,1}$ of minimal conformal weight, and we use $\pi_{2n+1}: V\rightarrow V(2n)\otimes\cL_{2n+1,1}$ to denote the $\cV ir$-module projection.

By assumption, the conformal-weight-$1$ space $V_{(1)}=V(2)\otimes v_{3,1}$ is three-dimensional, and as in \cite[Remark 8.9.1]{FLM} it is a Lie algebra with bracket $[u,v]=u_0 v$ and invariant symmetric bilinear form $\langle\cdot,\cdot\rangle$ defined by $u_1 v=\langle u,v\rangle\vac$ for $u,v\in V_{(1)}$. We claim that $\langle\cdot,\cdot\rangle$ is non-degenerate. Indeed, consider $a\otimes v_{3,1}, b\otimes v_{3,1}\in V_{(1)}$ for non-zero $a,b\in V(2)$. Then we have an identification of $V_1$-module intertwining operators
\begin{equation*}
\pi_1\circ Y_V\vert_{(a\otimes\cL_{3,1})\otimes(b\otimes\cL_{3,1})} =\langle a\otimes v_{3,1},b\otimes v_{3,1}\rangle\cY_{33}^1
\end{equation*}
where $\cY_{33}^1$ is the unique $V_1$-module intertwining operator of type $\binom{\cL_{1,1}}{\cL_{3,1}\,\cL_{3,1}}$ such that
\begin{equation*}
\cY_{33}^1(v_{3,1},x)v_{3,1}\in x^{-2}(\vac+x\cL_{1,1}[[x]]).
\end{equation*}
Since there is a non-zero intertwining operator of type $\binom{\cL_{1,1}}{\cL_{r,1}\,\cL_{3,1}}$ only for $r=3$, it follows that if $\langle a\otimes v_{3,1}, b\otimes v_{3,1}\rangle=0$ for all $a\in V(2)$, then the $V$-submodule generated by $b\otimes v_{3,1}$ is contained in $\bigoplus_{n= 1}^\infty V(2n)\otimes\cL_{2n+1,1}$. Since $V$ is simple, this is impossible for $b\neq 0$, and thus for such $b$ there exists $a\in V(2)$ such that $\langle a\otimes v_{3,1}, b\otimes v_{3,1}\rangle\neq 0$. This proves that $\langle\cdot,\cdot\rangle$ is non-degenerate.

Since $\langle\cdot,\cdot\rangle$ is non-degenerate, we can take $h\in V_{(1)}$ such that $\langle h,h\rangle =1$. We would like to show that $\omega=L_{-2}\vac$ is a multiple of $h_{-1}^2\vac$. Because $h_{2n+1,1}=n^2>2$ if $n>1$, we have
\begin{equation*}
h_{-1}^2\vac=c\cdot\omega+L_{-1}u
\end{equation*}
for some $c\in\CC$ and $u\in V_{(1)}$. We first claim that $u=0$. To show this, observe that vertex operator modes acting on $V$ satisfy
\begin{align*}
2\sum_{m=0}^\infty h_{-m}h_{m+1} & =\mathrm{Res}_x\,x^2 Y_V(h_{-1}^2\vac,x)\nonumber\\
& =c\cdot\mathrm{Res}_x\,x^2 Y_V(\omega,x)+\mathrm{Res}_x\,x^2\frac{d}{dx}Y_V(u,x) =c\cdot L_1-2 u_1.
\end{align*}
The left side of this equation annihilates $V_{(1)}$ because
\begin{equation*}
h_0h_1 v =\langle h,v\rangle h_0\vac =0
\end{equation*}
for all $v\in V_{(1)}$, and $L_1$ annihilates $V_{(1)}$ as well since $V_{(1)}$ consists of Virasoro primary vectors. Thus $u_1$ also annihilates $V_{(1)}$, that is, $\langle u,v\rangle=0$ for all $v\in V_{(1)}$. Since $\langle\cdot,\cdot\rangle$ is non-degenerate, this proves the claim.

We now have $h_{-1}^2\vac=c\cdot\omega$, and it remains to calculate $c$. Let $(\cdot,\cdot)$ be the unique non-degenerate invariant bilinear form on $\cL_{1,1}\subseteq V$ such that $(\vac,\vac)=1$. Thus
\begin{equation*}
(L_{-2}\vac,L_{-2}\vac)=(\vac,L_2 L_{-2}\vac)=\frac{1}{2},
\end{equation*}
so that using \eqref{eqn:Vir_comm_form} and the $L_{-1}$-derivative property,
\begin{align*}
\frac{c}{2} & = (L_{-2}\vac,c\cdot L_{-2}\vac) =(L_{-2}\vac,h_{-1}^2\vac)\nonumber\\
& =\mathrm{Res}_x\,x^{-1}(L_{-2}\vac,(\pi_1\circ Y_V)(h,x)h) =\mathrm{Res}_x\,x^{-1}(\vac,L_2(\pi_1\circ Y_V)(h,x)h\rangle\nonumber\\
& =\mathrm{Res}\,x\left(\vac, x\frac{d}{dx} (\pi_1\circ Y_V)(h,x)h+3(\pi_1\circ Y_V)(L_0 h,x)h\right)\nonumber\\
& =(\vac, -2h_1 h+3h_1 h) =\langle h,h\rangle(\vac,\vac)=1.
\end{align*}
Thus $c=2$ and we get $\omega =\frac{1}{2}h_{-1}^2\vac$.

We will now use \cite[Corollary 3.15]{LX} to show that $V$ is isomorphic to the lattice vertex operator algebra $V_L$ where $L$ consists of all $\alpha\in\CC$ such that for some non-zero $v\in V$, $h_0 v=\alpha v$. To apply this result, we need to verify Conditions (1) -- (3) of \cite{LX}: Condition (1) is satisfied because
\begin{equation*}
L_n h=\delta_{n,0} h,\qquad h_n h=\delta_{n,1}\vac
\end{equation*}
for $n\geq 0$. Condition (2) is that the algebra generated by the operators $h_n$, $n\geq 0$, acts locally finitely on $V$. This is clear because $V$ is a vertex operator algebra. Condition (3) is that $G_0 =\CC\vac$, where $G_0$ is the intersection of the generalized $h_0$-eigenspace with generalized eigenvalue $0$ with the Heisenberg vacuum space
\begin{equation*}
G=\lbrace v\in V\,\vert\,h_nv =0\,\,\text{for}\,\,n>0\rbrace.
\end{equation*}
Indeed, if $v\in G$, then $L_0 v=\frac{1}{2} h_0^2 v$, so if $v$ is additionally a generalized eigenvector for $h$ with generalized eigenvalue $0$, then $L_0^N v=0$ for some $N\in\ZZ_+$. Thus $v$ has conformal weight $0$ and it follows that $G_0=\CC\vac$.

It now follows from \cite[Corollary 3.15]{LX} that $V$ is isomorphic to a rank-one lattice vertex operator algebra $V_L$. Since $V$ has the same character as the $\mathfrak{sl}_2$-root lattice vertex operator algebra, which in turn is isomorphic to $L_1(\mathfrak{sl}_2)$, it follows that $V$ is isomorphic to $L_1(\mathfrak{sl}_2)$.
\end{proof}

We now consider what happens when a simple vertex operator algebra with decomposition \eqref{eqn:L1_sl2_decomp} does have a $PSL_2(\CC)$-action:
\begin{thm}\label{thm:2nd:L1_sl_2_unique}
Suppose $V$ is a simple vertex operator algebra extension of $V_1$ with an action of $PSL_2(\CC)$ by vertex operator algebra automorphisms such that $V\cong\bigoplus_{n=0}^\infty V(2n)\otimes\cL_{2n+1,1}$ as a $PSL_2(\CC)\times V_1$-module. Then the vertex operator algebra isomorphism $L_1(\mathfrak{sl}_2)\rightarrow V$ guaranteed by Theorem \ref{thm:L1_sl_2_unique} can be chosen to be an isomorphism of $PSL_2(\CC)$-modules.
\end{thm}

\begin{proof}
 By assumption, the conformal-weight-$1$ space $V_{(1)}=V(2)\otimes v_{3,1}$ is a $PSL_2(\CC)$-module isomorphic to $\mathfrak{sl}_2$, so we may fix a standard basis $\lbrace e, f, h\rbrace$ for $V_{(1)}$ as a $PSL_2(\CC)$-module. Moreover, because $PSL_2(\CC)$ acts on $V$ by vertex operator algebra automorphisms, the map $V_{(1)}\otimes V_{(1)} \rightarrow V_{(1)}$ given by $u\otimes v\mapsto u_0 v$ is a $PSL_2(\CC)$-module homomorphism, and thus it must be a multiple of the Lie bracket on $\mathfrak{sl}_2$, that is,
 \begin{equation}\label{eqn:scalar_mult_of_sl_2}
e_0 f= c\cdot h,\qquad f_0 h=2c\cdot f,\qquad h_0 e=2c\cdot e
\end{equation}
for some $c\in\CC$. Since $V_{(1)}$ is also a Lie algebra with bracket $[u,v]=u_0 v$, and since this Lie algebra is isomorphic to $\mathfrak{sl}_2$ by Theorem \ref{thm:L1_sl_2_unique}, we have $c\neq 0$ (but note that we cannot immediately identify the $\lbrace e, f, h\rbrace$ basis for $V_{(1)}$ considered as a $PSL_2(\CC)$-module with the corresponding basis for $V_{(1)}$ as a Lie algebra).

We can now replace the vertex operator $Y_{V}$ with $\varphi^{-1}\circ Y_{V}\circ(\varphi\otimes\varphi)$, where $\varphi$ is the linear isomorphism that acts as 
the identity on $V(2n)\otimes\cL_{2n+1,1}$ for $n\neq 1$ and by the scalar $c$ on $V(2)\otimes\cL_{3,1}$. Note that
\begin{equation*}
\varphi: (V,\varphi^{-1}\circ Y_{V}\circ(\varphi\otimes\varphi),\vac,\omega) \longrightarrow (V, Y_{V}, \vac,\omega)
\end{equation*}
is a vertex operator algebra isomorphism that is also a $PSL_2(\CC)$-module isomorphism. Thus we may assume that the simple vertex operator algebra structure on $V$ satisfies \eqref{eqn:scalar_mult_of_sl_2} with $c=1$. As in the proof of Theorem \ref{thm:L1_sl_2_unique},
 $(u,v)\rightarrow u_1 v$ defines a non-degenerate invariant bilinear form on $V_{(1)}\cong\mathfrak{sl}_2$:
\begin{equation*}
u_1 v=k\langle u,v\rangle\vac
\end{equation*}
for some non-zero $k\in\CC$, where $\langle\cdot,\cdot\rangle$ is the invariant bilinear form on $\mathfrak{sl}_2$ such that $\langle h,h\rangle =2$. Thus $h_1 h=2k\vac$, so by the proof of Theorem \ref{thm:L1_sl_2_unique}, $\omega=\frac{1}{4k} h_{-1}^2\vac$, and then
\begin{equation*}
L_0 =\frac{1}{4k}h_0^2 +\frac{1}{2k}\sum_{n=1}^\infty h_{-n} h_n.
\end{equation*}
Using $h_0 e=2e$, $h_1 e=k\langle h,e\rangle\vac=0$, and $h_n e=0$ for $n\geq 2$, we get $e=L_0 e=\frac{1}{k}e$, and thus $k=1$.

We can now conclude that there is a vertex algebra homomorphism from the universal affine vertex algebra $V^1(\mathfrak{sl}_2)$ at level $1$ to $V$ which sends the generators $e(-1)\vac$, $f(-1)\vac$, and $h(-1)\vac$ of $V^1(\mathfrak{sl}_2)$ to $e$, $f$, and $h$, respectively. Since $V$ is simple and generated by $e$, $f$, and $h$ (because $V$ is abstractly isomorphic to $L_1(\mathfrak{sl}_2)$), this homomorphism descends to a vertex operator algebra isomorphism $L_1(\mathfrak{sl}_2)\rightarrow V$. By construction, this isomorphism commutes with the $PSL_2(\CC)$-actions on the weight-one subspaces of $L_1(\mathfrak{sl}_2)$ and $V$. Then because both algebras are generated by their weight-one subspaces, and because $PSL_2(\CC)$ acts by automorphisms on both algebras, the isomorphism $L_1(\mathfrak{sl}_2)\rightarrow V$ is an isomorphism of $PSL_2(\CC)$-modules, as desired.
\end{proof}

The vertex operator algebra $L_1(\mathfrak{sl}_2)$, has a unique (up to isomorphism) irreducible module which is not isomorphic to $L_1(\mathfrak{sl}_2)$. We denote this module by $X$. The action of $\mathfrak{sl}_2$ on $X$ by zero-modes of the vertex operators for elements of $L_1(\mathfrak{sl}_2)_{(1)}\cong\mathfrak{sl}_2$ exponentiates to an action of $SL_2(\CC)$ on $X$ such that
\begin{equation}\label{eqn:sl_2_compat_1}
g\cdot Y_{X}(v,x)w = Y_{X}(g\cdot v,x)g\cdot w
\end{equation}
for $g\in SL_2(\CC)$, $v\in L_1(\mathfrak{sl}_2)$, and $w\in X$. In particular, the $SL_2(\CC)$-action and $V_1$-action on $X$ commute; we have
\begin{equation*}
X=\bigoplus_{n=1}^{\infty} V(2n-1)\otimes\cL_{2n,1}
\end{equation*}
as an $SL_2(\CC)\times V_1$-module. We now have a uniqueness result for the $SL_2(\CC)$-action on $X$:
\begin{thm}\label{thm:mod_for_L1_sl_2_unique}
Suppose $W$ is a simple $L_1(\mathfrak{sl}_2)$-module with an $SL_2(\CC)$-action such that
\begin{equation}\label{eqn:sl_2_compat_2}
g\cdot Y_W(v,x)w= Y_W(g\cdot v,x)g\cdot w
\end{equation}
for all $g\in SL_2(\CC)$, $v\in L_1(\mathfrak{sl}_2)$, and $w\in W$, and such that $W\cong\bigoplus_{n=1}^\infty V(2n-1)\otimes \cL_{2n,1}$ as an $SL_2(\CC)\times V_1$-module. Then there is an $L_1(\mathfrak{sl}_2)$-module isomorphism $X\rightarrow W$ which is also an $SL_2(\CC)$-module isomorphism.
\end{thm}

\begin{proof}
Let $T(W)$ denote the top level of $W$, that is, the lowest conformal weight space $V(1)\otimes v_{2,1}$. The assumed compatibility of the $SL_2(\CC)$- and $L_1(\mathfrak{sl}_2)$-actions on $W$ implies that the linear map $L_1(\mathfrak{sl}_2)_{(1)}\otimes T(W)\rightarrow T(W)$ give by $v\otimes w\mapsto v_0 w$ is an $SL_2(\CC)$-module homomorphism, and thus must be a multiple of the $\mathfrak{sl}_2$-action on $V(1)$. That is, if $\lbrace e, f, h\rbrace$ is the standard basis of $L_1(\mathfrak{sl}_2)_{(1)}\cong\mathfrak{sl}_2$, then there is a basis $\lbrace v_1, v_{-1}\rbrace$ of $T(W)$ such that
\begin{equation}\label{eqn:sl2_action_on_TW}
\begin{array}{lll}
e_0 v_1 =0, & h_0 v_1 =c\cdot v_1, & f_0 v_1=c\cdot v_{-1},\\
e_0 v_{-1} = c\cdot v_1, & h_0 v_{-1} =-c\cdot v_{-1}, & f_0 v_{-1} =0\\
\end{array}
\end{equation}
for some $c\in\CC$.

On the other hand, the axioms for vertex operator algebras and modules show that the action of $L_1(\mathfrak{sl}_2)_{(1)}$ on $T(W)$ by zero-modes gives $T(W)$ the structure of an $\mathfrak{sl}_2$-module. The formulas \eqref{eqn:sl2_action_on_TW} describe an $\mathfrak{sl}_2$-module structure only when $c=0$ or $c=1$, since if $c\neq 0$, then $v_1$ is the unique (up to scale) highest-weight vector in $T(W)$ and thus must have $h$-eigenvalue $1$ as $T(W)$ is two-dimensional. In fact, we may conclude $c\neq 0$ because the classification of simple $L_1(\mathfrak{sl}_2)$-modules shows that $W$ must be isomorphic to $X$ as an $L_1(\mathfrak{sl}_2)$-module, and we know that $L_1(\mathfrak{sl}_2)_{(1)}$ acts non-trivially on the top level of $X$. Thus \eqref{eqn:sl2_action_on_TW} holds with $c=1$, showing that the assumed $SL_2(\CC)$-action on $T(W)$ is the same as that obtained by exponentiating the zero-mode action of $L_1(\mathfrak{sl}_2)_{(1)}\cong\mathfrak{sl}_2$ on $T(W)$.

We have now shown that any $L_1(\mathfrak{sl}_2)$-module isomorphism $X\rightarrow W$ restricts to an $SL_2(\CC)$-module isomorphism on top levels, since the $SL_2(\CC)$-actions on both top levels are obtained by exponentiating the zero-mode actions of $L_1(\mathfrak{sl}_2)_{(1)}$. Then because both $X$ and $W$ are generated by their top levels, and because \eqref{eqn:sl_2_compat_1} and \eqref{eqn:sl_2_compat_2} both hold, we conclude that any $L_1(\mathfrak{sl}_2)$-module isomorphism $X\rightarrow W$ is also an $SL_2(\CC)$-module isomorphism.
\end{proof}

\section{Uniqueness of the chiral universal centralizer}\label{app:chiral_univ_cent_unique}

In this appendix, we prove the uniqueness assertion in Theorem \ref{thm:glue} using Theorems \ref{thm:2nd:L1_sl_2_unique} and \ref{thm:mod_for_L1_sl_2_unique}. Thus let $V$ be any simple conformal vertex algebra containing $V_1\otimes V_{25}$ as a vertex operator subalgebra and such that 
\begin{equation*}
V \cong \bigoplus_{r\in\ZZ_+} \cL_{r,1}^{(1)}\otimes\cL_{r,1}^{(25)}
\end{equation*}
as a $V_1\otimes V_{25}$-module. The $\mathfrak{sl}_2$-type fusion rules for $V_1$- and $V_{25}$-modules imply that $V$ has a conformal vertex algebra involution which acts as $(-1)^{r-1}$ on $\cL_{r,1}^{(1)}\otimes\cL_{r,1}^{(25)}$. Let $V^0$ be the fixed-point subalgebra of this involution and $V^1$ the eigenspace with eigenvalue $-1$, so that
\begin{equation*}
V^0 \cong \bigoplus_{n=0}^\infty \cL_{2n+1,1}^{(1)}\otimes\cL_{2n+1,1}^{(25)},\qquad V^1\cong\bigoplus_{n = 1}^\infty \cL_{2n,1}^{(1)}\otimes\cL_{2n,1}^{(25)}
\end{equation*}
as $V_1\otimes V_{25}$-modules. It is easy to see that $V^0$ is a simple conformal vertex algebra as follows: For any non-zero $v\in V^0$,
\begin{equation*}
V =\mathrm{span}\lbrace u_n v\,\vert\,u\in V, n\in\ZZ\rbrace = \mathrm{span}\lbrace u_n v\,\vert\,u\in V^0, n\in\ZZ\rbrace +\mathrm{span}\lbrace u_n v\,\vert\,u\in V^1, n\in\ZZ\rbrace
\end{equation*}
by \cite[Proposition 4.5.6]{LL} since $V$ is simple. But since $V=V^0\oplus V^1$ and since
\begin{equation*}
\mathrm{span}\lbrace u_n v\,\vert\,u\in V^i, n\in\ZZ\rbrace\subseteq V^i
\end{equation*}
for $i=0,1$, it follows that
\begin{equation*}
V^0=\mathrm{span}\lbrace u_n v\,\vert\,u\in V^0, n\in\ZZ\rbrace
\end{equation*}
for any non-zero $v\in V^0$, and thus $V^0$ is simple. A similar argument shows that $V^1$ is a simple $V^0$-module.

\begin{prop}\label{prop:V0_V1_unique}
Any simple conformal vertex algebra extension of $V_1\otimes V_{25}$ which is isomorphic to $V^0$ as a $V_1\otimes V_{25}$-module is isomorphic to $V^0$ as a conformal vertex algebra. Moreover, any simple $V^0$-module which is isomorphic to $V^1$ as a $V_1\otimes V_{25}$-module is isomorphic to $V^1$ as a $V^0$-module.
\end{prop}
\begin{proof}
From Proposition \ref{prop: tens-equiv}, $\cO_{25}^0$ is braided tensor equivalent to $(\rep\mathfrak{sl}_2)^\tau$ with a suitable braiding. Thus using \cite[Theorem 7.5]{CMY1}, uniqueness of the simple conformal vertex algebra structure on $V^0=\bigoplus_{n=0}^\infty \cL_{2n+1,1}^{(1)}\otimes\cL_{2n+1,1}^{(25)}$ extending $V_1\otimes V_{25}$ is equivalent to uniqueness of the simple commutative algebra structure on the object
\begin{equation*}
A^0=\bigoplus_{n=0}^\infty V(2n)\otimes\cL_{2n+1,1}^{(1)}
\end{equation*}
in the Ind-category $\ind((\rep\mathfrak{sl}_2)^\tau \otimes \cO_1^0)$. Moreover, uniqueness of the simple $V^0$-module structure on $V^1$ is equivalent to uniqueness of the simple $A^0$-module structure on the object
\begin{equation*}
A^1 =\bigoplus_{n=1}^\infty V(2n-1)\otimes\cL_{2n,1}^{(1)}
\end{equation*}
in $\ind((\rep\mathfrak{sl}_2)^\tau\otimes\cO_{1}^0)$.

Since the $\mathfrak{sl}_2$-modules $V(2n)$ for $n\in\NN$ are objects of the maximal symmetric tensor subcategory $\rep PSL_2(\CC)\subseteq(\rep\mathfrak{sl}_2)^\tau$, we can use \cite[Theorem 7.5]{CMY1} again to see that a simple commutative algebra structure on $A^0$ in $\ind((\rep\mathfrak{sl}_2)^\tau\otimes\cO_1^0)$ amounts to a vertex operator algebra structure $(A^0, Y_{A^0},\vac,\omega)$ such that:
\begin{enumerate}
\item The $PSL_2(\CC)$-action on the $V(2n)$ factors in $A^0$ gives a $PSL_2(\CC)$-action by vertex operator algebra automorphisms on $A^0$.

\item The only $PSL_2(\CC)$-invariant ideals of $A^0$ are $0$ and $A^0$.
\end{enumerate}
Moreover, uniqueness of the simple commutative algebra structure on $A^0$ in $\ind((\rep\mathfrak{sl}_2)^\tau\otimes\cO_1^0)$ amounts to uniqueness of the vertex operator algebra structure on $A^0$ up to an isomorphism that is also a $PSL_2(\CC)$-module isomorphism.

Similarly, because $(\rep\mathfrak{sl}_2)^\tau$ is equivalent to $\rep\mathfrak{sl}_2$ when considered as a module category for $\rep PSL_2(\CC)$, a simple $A^0$-module structure on $A^1$ (where $A^0$ is considered as a commutative algebra in $\ind((\rep\mathfrak{sl}_2)^\tau\otimes\cO_1^0)$) is equivalent to an $A^0$-module structure $(A^1, Y_{A^1})$ (where $A^0$ is considered as a vertex operator algebra) such that:
\begin{enumerate}

\item[(3)] The $SL_2(\CC)$-action on the $V(2n-1)$ factors in $A^1$ and the $PSL_2(\CC)$-action on $A^0$ are compatible in the sense that
\begin{equation*}
g\cdot Y_{A^1}(v,x)w =Y_{A^1}(g\cdot v,x)g\cdot w
\end{equation*}
for all $g\in SL_2(\CC)$, $v\in A^0$, and $w\in A^1$.

\item[(4)] The only $SL_2(\CC)$-invariant $A^0$-submodules of $A^1$ are $0$ and $A^1$.

\end{enumerate}
Moreover, uniqueness of the $A^0$-module structure on $A^1$ (where $A^1$ is considered as a commutative algebra in $\ind((\rep\mathfrak{sl}_2)^\tau \otimes \cO_1^0)$) is equivalent to uniqueness of the $A^0$-module structure on $A^1$ (where $A^0$ is considered as a vertex operator algebra) up to an $A^0$-module isomorphism that is also an $SL_2(\CC)$-module isomorphism.

We claim that Conditions (1) and (2) imply that $A^0$ is a simple vertex operator algebra, and we claim that $A^1$ is a simple $A^0$-module. Then the desired uniqueness assertions will follow from Theorems \ref{thm:2nd:L1_sl_2_unique} and \ref{thm:mod_for_L1_sl_2_unique}: $A^0$ will be isomoprhic to the simple affine vertex operator algebra $L_1(\mathfrak{sl}_2)$ with its standard $PSL_2(\CC)$-action by automorphisms, and $A^1$ will be isomorphic to the unique non-trivial simple $L_1(\mathfrak{sl}_2)$-module with its standard $SL_2(\CC)$-action.

To prove the claims, first let $I\subseteq A^0$ be a non-zero ideal. Since $A^0$ is a semisimple $\cV ir$-module, so is $I$, and thus $I$ contains $v\otimes\cL_{2n+1,1}^{(1)}$ for some $n\in\NN$ and non-zero $v\in V(2n)$. Now consider the ideal generated by $V(2n)\otimes\cL_{2n+1,1}^{(1)}$; it is $PSL_2(\CC)$-invariant by Condition (1), and thus it equals $A^0$ by Condition (2). In particular, the ideal generated by $V(2n)\otimes\cL_{2n+1,1}^{(1)}$ contains $V(0)\otimes\cL_{1,1}^{(1)}$. Since there is a non-zero $PSL_2(\CC)$-module homomorphism $V(2m)\otimes V(2n)\rightarrow V(0)$ if and only if $m=n$, it follows that the Virasoro intertwining operator
\begin{equation*}
\pi_1\circ Y_{A^0}\vert_{(V(2n)\otimes\cL_{2n+1,1}^{(1)})\otimes(V(2n)\otimes\cL_{2n+1,1}^{(1)})},
\end{equation*}
where $\pi_1: A^0\rightarrow V(0)\otimes\cL_{1,1}^{(1)}$ is the projection, is non-zero. Thus this intertwining operator has the form $b\otimes\cY$ where $b: V(2n)\otimes V(2n)\rightarrow V(0)$ is a non-zero (and thus non-degenerate) invariant bilinear form and $\cY$ is a non-zero intertwining operator of type $\binom{\cL_{1,1}^{(1)}}{\cL_{2n+1,1}^{(1)}\,\cL_{2n+1,1}^{(1)}}$.

Now recall that our non-zero ideal $I\subseteq A^0$ contains $v\otimes\cL_{2n+1,1}^{(1)}$ for some non-zero $v\in V(2n)$. Thus letting $v_{2n+1,1}$ denote a non-zero lowest-conformal-weight vector in $\cL_{2n+1,1}^{(1)}$, we have now shown that there is some $v'\in V(2n)$ such that
\begin{equation*}
Y_{A^0}(v'\otimes v_{2n+1,1},x)(v\otimes v_{2n+1,1})\in I((x))
\end{equation*}
has non-zero projection to $V(0)\otimes\cL_{1,1}^{(1)}$. Then because $I$ is a semisimple $\cV ir$-module, this shows that $I$ contains $V(0)\otimes\cL_{1,1}^{(1)}$ and thus contains the vacuum vector $\vac$. Since $\vac$ generates $A^0$ as an $A^0$-module, we conclude that $I=A^0$, proving $A^0$ is simple.

To show that $A^1$ is a simple $A^0$-module, we can now use Theorem \ref{thm:L1_sl_2_unique}, which shows that $A^0$ is isomorphic to $L_1(\mathfrak{sl}_2)$ as a vertex operator algebra. Then $A^1$ must be simple because the only $L_1(\mathfrak{sl}_2)$-module with the same $V_1$-module decomposition as $A^1$ is the non-trivial simple $L_1(\mathfrak{sl}_2)$-module.
\end{proof}

We can now prove uniqueness of the simple conformal vertex algebra structure on $V^0\oplus V^1$, and thus the uniqueness assertion in Theorem \ref{thm:glue}:
\begin{proof}
Let $(V,Y_V,\vac,\omega)$ and $(V,\til{Y}_V,\vac,\omega)$ be two simple conformal vertex algebra structures on $V$ such that
\begin{equation*}
Y_V\vert_{(\cL^{(1)}_{1,1}\otimes\cL_{1,1}^{(25)})\otimes V} =\til{Y}_V\vert_{(\cL^{(1)}_{1,1}\otimes\cL_{1,1}^{(25)})\otimes V}.
\end{equation*}
Since $V=V^0\oplus V^1$ where $V^0$ is a simple subalgebra and $V^1$ is a simple $V^0$-module with respect to either conformal vertex algebra structure on $V$, Proposition \ref{prop:V0_V1_unique} shows that we may assume
\begin{equation*}
Y_V\vert_{V^0\otimes V} = \til{Y}_V\vert_{V^0\otimes V}.
\end{equation*}
Then skew-symmetry dictates that 
\begin{equation*}
Y_V\vert_{V^1\otimes V^0}=e^{xL_{-1}}Y_V(\cdot,-x)\cdot\vert_{V^0\otimes V^1} =e^{xL_{-1}}\til{Y}_V(\cdot,-x)\cdot\vert_{V^0\otimes V^1}=\til{Y}_V\vert_{V^1\otimes V^0},
\end{equation*}
so it remains to consider $\til{Y}_V\vert_{V^1\otimes V^1}$.

The $\mathfrak{sl}_2$-type fusion rules of $V_1$- and $V_{25}$-modules show that $\mathrm{Im}\,\til{Y}_V\vert_{V^1\otimes V^1}\subseteq V^0$. We claim the space of $V^0$-module intertwining operators of type $\binom{V^0}{V^1\,V^1}$ is one dimensional, so that
\begin{equation*}
\til{Y}_V\vert_{V^1\otimes V^1}=\alpha\cdot Y_V\vert_{V^1\otimes V^1}
\end{equation*}
for some $\alpha\in\CC$ (note that $Y_V\vert_{V^1\otimes V^1}\neq 0$ since otherwise $V^1$ would be a proper ideal of $(V,Y_V,\vac,\omega)$). Assuming the claim, we have $\alpha\neq 0$ since $(V,\til{Y}_V,\vac,\omega)$ is a simple conformal vertex algebra, so then
\begin{equation*}
\sigma :=\Id_{V^0}\oplus\sqrt{\alpha}\cdot\Id_{V^1}
\end{equation*}
is a linear isomorphism of $V$ such that $\sigma\circ \til{Y}_V=Y_V\circ(\sigma\otimes\sigma)$, showing that the two simple conformal vertex algebra structures on $V$ are isomorphic (this is a special case of the uniqueness proof for simple current extensions in \cite[Proposition 5.3]{DM}).

To prove the claim, let $q_r: \cL_{r,1}^{(1)}\otimes\cL_{r,1}^{(25)}\rightarrow V$ and $\pi_r: V\rightarrow\cL_{r,1}^{(1)}\otimes\cL_{r,1}^{(25)}$ for $r\in\ZZ_+$ denote the natural inclusion and projection, respectively. It is enough to show that the linear map
\begin{equation*}
\cY\mapsto\pi_1\circ\cY\circ(q_{2}\otimes q_{2})
\end{equation*}
from the space of $V^0$-module intertwining operators of type $\binom{V^0}{V^1\,V^1}$ to the one-dimensional space of $V_1\otimes V_{25}$-module intertwining operators of type $\binom{\cL_{1,1}^{(1)}\otimes\cL_{1,1}^{(25)}}{\cL_{2,1}^{(1)}\otimes\cL_{2,1}^{(25)}\,\cL_{2,1}^{(1)}\otimes\cL_{2,1}^{(25)}}$ is injective. Indeed, let
\begin{align*}
\cY: V^1\otimes V^1 & \rightarrow V^0((x))\nonumber\\
u_1\otimes v_1 & \mapsto\cY(u_1,x)v_1=\sum_{n\in\ZZ} (u_1)_n v_1\,x^{-n-1}
\end{align*}
be a non-zero intertwining operator, and fix a non-zero $v_1\in\cL_{2,1}^{(1)}\otimes\cL_{2,1}^{(25)}$. Then
\begin{equation*}
\mathrm{span}\lbrace (u_1)_n v_1\,\vert\,u_1\in V^1, n\in\ZZ\rbrace
\end{equation*}
is a $V^0$-submodule of $V^0$ by the easy intertwining operator generalization of \cite[Proposition 4.5.7]{LL}. Using the commutator formula for intertwining operators, we see that this submodule is non-zero, since $v_1$ generates $V^1$ as a $V^0$-module and $\cY$ is non-zero (see \cite[Proposition 11.9]{DL}). So in fact
\begin{equation*}
\cL_{1,1}^{(1)}\otimes\cL_{1,1}^{(25)} =\pi_1(V^0) =\pi_1(\mathrm{span}\lbrace (u_1)_n v_1\,\vert\,u_1\in V^1, n\in\ZZ\rbrace)
\end{equation*}
since $V^0$ is simple. That is, there exists $m\in\ZZ_+$ such that
\begin{equation*}
\pi_1\circ\cY\circ(q_{2m}\otimes q_2)\neq 0.
\end{equation*}
Finally, the $\mathfrak{sl}_2$-type fusion rules of $V_1$- and $V_{25}$-modules force $m=1$, completing the proof of the claim and also of the theorem.
\end{proof}

\end{document}